\documentclass{article}
\pdfoutput=1
\usepackage[utf8]{inputenc}
\usepackage{amsmath}
\usepackage{amssymb}
\usepackage{multirow, hhline}
\usepackage[table]{xcolor}
\usepackage[para,online,flushleft]{threeparttable}
\usepackage{microtype}   
\usepackage[letterpaper,margin=1.0in]{geometry}
\usepackage{graphicx}

\usepackage{hyperref}       
\usepackage{url}            
\usepackage{booktabs}       
\usepackage{amsfonts}       
\usepackage{bbding}
\usepackage{xcolor}         
\usepackage[numbers,sort]{natbib}
\usepackage{color}
\usepackage{subcaption}
\usepackage{graphicx}
\usepackage{tablefootnote}
\usepackage{algorithm}
\usepackage{algpseudocode}
\usepackage{bbding}
\usepackage{footnote}
\makesavenoteenv{tabular}
\usepackage{threeparttable}
\usepackage{enumitem}

\usepackage{amsmath}\allowdisplaybreaks
\usepackage{amsfonts,bm}
\usepackage{amssymb}

\def\eqref#1{equation~(\ref{#1})}

\def\ceil#1{\lceil #1 \rceil}
\def\floor#1{\left\lfloor #1 \right\rfloor}
\def\1{\bf{1}}

\newcommand{\Norm}[1]{\left\| #1 \right\|}

\def\inner#1#2{\langle #1, #2 \rangle}

\def\vzero{{\bf{0}}}

\def\va{{\bf{a}}}
\def\vb{{\bf{b}}}

\def\ve{{\bf{e}}}
\def\vf{{\bf{f}}}
\def\vg{{\bf{g}}}

\def\vu{{\bf{u}}}
\def\vv{{\bf{v}}}

\def\vx{{\bf{x}}}
\def\vy{{\bf{y}}}
\def\vz{{\bf{z}}}

\def\fH{{\mathcal{H}}}

\def\fO{{\mathcal{O}}}

\def\fS{{\mathcal{S}}}

\def\BN{{\mathbb{N}}}

\def\BR{{\mathbb{R}}}

\def\mD {{\bf D}}

\def\mG {{\bf G}}
\def\mH {{\bf H}}
\def\mI {{\bf I}}
\def\mJ {{\bf J}}

\def\mP {{\bf P}}
\def\mQ {{\bf Q}}

\def\mU {{\bf U}}
\def\mV {{\bf V}}

\usepackage{amsthm}
\usepackage{etoolbox}
\theoremstyle{plain}
\newtheorem{thm}{Theorem}
\AfterEndEnvironment{thm}{\noindent\ignorespaces}
\newtheorem{dfn}{Definition}
\AfterEndEnvironment{defn}{\noindent\ignorespaces}

\AfterEndEnvironment{exmp}{\noindent\ignorespaces}
\newtheorem{lem}{Lemma}
\AfterEndEnvironment{lem}{\noindent\ignorespaces}
\newtheorem{asm}{Assumption}
\AfterEndEnvironment{asm}{\noindent\ignorespaces}

\AfterEndEnvironment{remark}{\noindent\ignorespaces}
\newtheorem{cor}{Corollary}
\AfterEndEnvironment{cor}{\noindent\ignorespaces}
\newtheorem{prop}{Proposition}
\AfterEndEnvironment{prop}{\noindent\ignorespaces}

\makeatletter
\def\Ddots{\mathinner{\mkern1mu\raise\p@
\vbox{\kern7\p@\hbox{.}}\mkern2mu
\raise4\p@\hbox{.}\mkern2mu\raise7\p@\hbox{.}\mkern1mu}}
\makeatother

\makeatletter
\newcommand*{\rom}[1]{\expandafter\@slowromancap\romannumeral #1@}
\makeatother

\usepackage{tcolorbox}
\usepackage{pifont}
\definecolor{mydarkgreen}{RGB}{39,130,67}
\definecolor{mydarkred}{RGB}{192,25,25}
\definecolor{bgcolor}{rgb}{0.93,0.99,1}
\definecolor{bgcolor2}{rgb}{0.8,1,0.8}
\definecolor{bgcolor3}{rgb}{0.50,0.90,0.50}

\usepackage{relsize} 

\def\IGN{\textnormal{\texttt{IGN}}}
\def\MBIGN{\textnormal{\texttt{MB-IGN}}}
\begin{document}

\title{Incremental Gauss--Newton Methods with\\ Superlinear Convergence Rates}

\author{
    Zhiling Zhou\thanks{School of Data Science, Fudan University; zlzhou20@fudan.edu.cn} \qquad \quad
    Zhuanghua Liu\thanks{Department of Computer Science, National University of Singapore; liuzhuanghua9@gmail.com} \qquad \quad Chengchang Liu\thanks{Department of Computer Science and Engineering, The Chinese University of Hong Kong; ccliu22@cse.cuhk.edu.hk} \qquad \quad Luo Luo\thanks{School of Data Science, Fudan University; luoluo@fudan.edu.cn}
 }
 \date{}
 

\maketitle

\begin{abstract}
This paper addresses the challenge of solving large-scale nonlinear equations with H\"older continuous Jacobians.
We introduce a novel Incremental Gauss--Newton (IGN) method within explicit superlinear convergence rate, which outperforms existing 
methods that only achieve linear convergence rate.
In particular, we formulate our problem by the nonlinear least squares with finite-sum structure, and our method incrementally iterates with the information of one component in each round.
We also provide a mini-batch extension to our IGN method that obtains an even faster superlinear convergence rate.
Furthermore, we conduct numerical experiments to show the advantages of the proposed methods.
\end{abstract}

\section{Introduction}
\vskip 0.2cm
We study the problem of solving the system of nonlinear equations
\begin{align}
\label{eq:ori_prob}
    \vf(\vx) = \vzero,
\end{align}
where the nonlinear vector function $\vf:\BR^d\to\BR^n$ is Lipschitz continuous and its Jacobian is H\"older continuous.
This formulation is a fundamental problem in scientific computing \cite{nesterov2006cubic}, and it is popular in a large number of applications including
 machine learning \cite{defossez2015averaged, botev2017practical, bai2019deep}, control system \cite{berthier2021fast}, data assimilation \cite{tremolet2007model} and game theory~\cite{frehse1984nonlinear, nourian2013ϵ}.

The Newton-type methods \cite{dennis1996numerical, kelley1995iterative, kelley2003solving,ben1966newton,nocedal1999numerical,wang2012gauss} are widely used for solving nonlinear equations.
The classical Newton's method uses the curvature information in Jacobians to obtain a local quadratic convergence rate \cite{nocedal1999numerical}, while it suffers from the expensive computational cost to access the Jacobian and its (pseudo) inverse. 
Several lines of research focus on approximating Newton's methods with inexact Jacobians.
For example, the quasi-Newton methods~\cite{broyden1965class,li1999globally, al2014broyden,li1999globally} estimate the Jacobians via secant equations, leading to the iteration scheme that only needs to access the function value and Jacobian-vector calls.  
The explicit local superlinear convergence rates of these methods have been established in recent years~\cite{lin2021explicit, liu2022quasi, liu2023block, ye2021greedy}. 
Another line of work \cite{yuan2022sketched,ye2021approximate,pilanci2017newton} introduce matrix sketching technique \cite{woodruff2014sketching} to
reduce the dimension of the Jacobian matrix, which improves the computational efficiency per iteration.
The superiority of their local convergence depends on the structure of Jacobian in the specific problem. 
Although quasi-Newton and sketched Newton methods can benefit from the inexact Jacobians, they still require accessing the full information of the nonlinear vector function value at every iteration.

For large-scale nonlinear equations, we are interested in methods that do not require the computation of full function values and Jacobians.
In particular, \citet{bertsekas1996incremental}  proposed a variant of Gauss--Newton (GN) method by following the  Extended Kalman Filter (EKF) framework \cite{athans1968suboptimal, ljung1979asymptotic, bell1994iterated}, which incrementally accesses partial information of the vector function values and corresponding Jacobian during the iterations.
Consequently, \citet{moriyama2003incremental} incorporated a stepsize into the EKF method, which guarantees the global linear convergence rate under the gradient-growth condition~\cite{gurbuzbalaban2015globally}. 
In the past decade, the incremental (quasi) Newton methods with local superlinear convergence rates are established for strongly convex optimization \cite{rodomanovnewton, rodomanov2016superlinearly,mokhtari2018iqn,lahoti2023sharpened,liu2024incremental}.\footnote{In the view of solving nonlinear equations, the methods designed for convex optimization \cite{rodomanovnewton, rodomanov2016superlinearly,mokhtari2018iqn,lahoti2023sharpened,liu2024incremental} require an additional assumption that the Jacobian is symmetric positive-definite.} 
However, the superiority of local convergence for incremental Newton-type methods in solving the general nonlinear equations is still unclear.

In this work, we propose an incremental Gauss--Newton (\IGN) method for solving the systems of nonlinear equations.
Our method only requires access to one component of the nonlinear vector function and its gradient per iteration.
We maintain an aggregated vector and an aggregated matrix to estimate the vector function value and its Jacobian by incrementally updating.
We also introduce a Gram matrix with a low-rank update to reduce the computational cost of matrix inverse in vanilla Gauss--Newton methods.
The theoretical analysis shows our \IGN~method enjoys explicit local superlinear convergence rate for nonlinear equations problem with H\"older continuous Jacobians.
Furthermore, we provide a variant of our \IGN~that makes use of the information of a mini-batch of components, which achieves an even faster superlinear convergence rate. 
The numerical experiments on real-world applications validate the advantages of the proposed methods.

\paragraph{Paper Organization}
In Section~\ref{sec:preliminaries}, we formalize the notations and assumptions for our problem.
In Section~\ref{sec:ign}, we propose our incremental Gauss--Newton (\texttt{IGN}) method and present its convergence analysis.
In Section~\ref{sec:ign_k}, we extend the \texttt{IGN} method with the mini-batch update to obtain an even faster convergence rate.
In Section~\ref{sec:related_work}, we provide a discussion to compare the proposed method with related works.
In Section~\ref{sec:experiments}, we conduct numerical experiments to show the advantages of our methods.
We conclude our work in Section~\ref{sec:conclusion}.

\section{Preliminaries}\label{sec:preliminaries}
In this section, we formalize the notations and assumptions throughout this paper.

\subsection{Notations}
We let $[n]=\{1,\dots,n\}$ and use the notation $t\%n$ to present the remainder of $t$ divided by $n$.
We denote $\ve_i\in\BR^n$ as the $i$-th standard basis vector of the $n$-dimensional Euclidean space for all $i\in [n]$. 
We use $\Norm{\cdot}$ to represent the Euclidean norm for a given vector and the spectral norm for a given matrix. 
Moreover, we use the notation $\sigma_{\min}(\cdot)$ to represent the smallest singular value for a given matrix.

For the system of nonlinear equations (\ref{eq:ori_prob}), we partition the vector function $\vf:\BR^d\to\BR^n$ at $\vx\in\BR^d$ as~$\vf(\vx)=[f_1(\vx),\dots, f_n(\vx)]^\top\in\BR^n$, where $f_i:\BR^d\to\BR$.
We also denote the gradient of component $f_i(\cdot)$ at $\vx\in\BR^d$ as $\vg_i(\vx)=\nabla f_i(\vx)$, and we organize the corresponding Jacobian as~$\mJ(\vx)=[\vg_1(\vx), \cdots, \vg_n(\vx)]^\top\in\BR^{n\times d}$.

\subsection{Assumptions}

Throughout this paper, we suppose the function $\vf:\BR^d\to\BR^n$ satisfies the following assumptions.
\begin{asm}
\label{asm:L-f}
    We suppose the vector function $\vf:\BR^d\to\BR^n$ is Lipschitz continuous, i.e., there exists constant $L_f>0$ such that
    \begin{align}
        \Norm{\vf(\vx)-\vf(\vy)} \leq L_f\Norm{\vx - \vy}
    \end{align}
    for all $\vx,\vy\in\BR^d$.
\end{asm}
\begin{asm}\label{asm:holder-g}
    We suppose the Jacobian $\mJ:\BR^d\to\BR^{n\times d}$  is $\nu$-H\"older continuous for some $\nu\in(0,1]$, i.e., there exists constant $\fH_\nu>0$ such that
    \begin{align}
        \Norm{\mJ(\vx)-\mJ(\vy)} \leq \fH_\nu\Norm{\vx - \vy}^\nu
    \end{align}
    for all $\vx,\vy\in\BR^d$.
\end{asm}
\begin{asm}
\label{asm:b-J-sing}
    The system of the nonlinear equations (\ref{eq:ori_prob}) holds $n\geq d$ and has a non-degenerate solution $\vx^*\in\BR^d$, i.e., there exists some $\mu>0$ such that
    \begin{align}
    \label{eq:b-J-sing}
    \mu = \sigma_{\min} (\mJ(\vx^*))> 0.
    \end{align}
\end{asm}
Noticing that most of existing work \cite{rodomanov2016superlinearly,yuan2022sketched,liu2023block,bertsekas1997new,bertsekas1996incremental,moriyama2003incremental, gurbuzbalaban2015globally} focus on the assumption of Lipschitz continuous Jacobian, which is a special case of our Assumption \ref{asm:holder-g} by taking $\nu=1$.

\section{The Incremental Gauss--Newton Method} \label{sec:ign}

\begin{algorithm}[t]
\caption{Incremental Gauss–Newton Method (\texttt{IGN})}
\label{alg:IGN-1}
\begin{algorithmic}[1]
\State \textbf{Input:} $\vx^0\in\BR^d$, $\vu^0\in\BR^d$, $\mH^0,\mG^0\in\BR^{d\times d}$  \vskip0.1cm
\State \textbf{for} $t=0,1,\dots$ \vskip0.12cm
\State \quad  $\vx^{t+1}=\mG^t\vu^t$ \vskip0.12cm
\State \quad  $i_t =t\%n + 1$ \vskip0.12cm
\State \quad  $\mU^t = \big[
    - \vg_{i_t}(\vz_{i_t}^t),~~ 
    \vg_{i_t}(\vx^{t+1})
\big]$ \vskip0.12cm
\State \quad
$\mV^t = \big[\vg_{i_t}(\vz_{i_t}^t),~~ 
    \vg_{i_t}(\vx^{t+1})\big]$ \vskip0.12cm
\State \quad $\vu^{t+1} = \vu^t - \left(\vg_{i_t}(\vz_{i_t}^t)^\top\vz_{i_t}^{t} - f_{i_t}(\vz_{i_t}^t)\right)\vg_{i_t}(\vz_{i_t}^t) 
 + \left(\vg_{i_t}(\vx^{t+1})^\top\vx^{t+1} - f_{i_t}(\vx^{t+1})\right)\vg_{i_t}(\vx^{t+1})$ \vskip0.12cm
\State \quad $\mH^{t+1} =\mH^t - \vg_{i_t}(\vz_{i_t}^t)\vg_{i_t}(\vz_{i_t}^t)^\top+ \vg_{i_t}(\vx^{t+1})\vg_{i_t}(\vx^{t+1})^\top$ \vskip0.12cm
\State\label{line:update-G} \quad  $\mG^{t+1} =\mG^{t} - \mG^{t}\mU^t(\mI + (\mV^t)^\top\mG^{t}\mU^t)^{-1}(\mV^t)^\top\mG^{t}$ \vskip0.12cm
\State \quad  $\vz_i^{t+1} = \begin{cases}
        \vx^{t+1}, & \text{if~} i=i_t \\
        \vz_i^t, & \text{otherwise}
    \end{cases}$ \vskip0.12cm
\State \textbf{end for}
\end{algorithmic}
\end{algorithm}

In this section, we propose the Incremental Gauss-Newton (\IGN) method and provide its explicit superlinear convergence rate.

\subsection{The Algorithm}

We first introduce the intuition of our algorithm design.
Solving the system of nonlinear equations~(\ref{eq:ori_prob}) can be regarded as minimizing the norm of the nonlinear vector function $\vf:\BR^d\to\BR^n$, which means we can reformulate the problem as the following nonlinear least squares minimization problem
\begin{align}\label{eq:min_prob}
    \min_{\vx\in\BR^d} \phi(\vx)\triangleq \frac{1}{2}\sum_{i=1}^n (f_i(\vx))^2.
\end{align}
For each component $f_i:\BR^d\to\BR$, we consider its linear approximation
\begin{align}\label{eq:approx-fi}
    f_i(\vx) \approx f_i(\vz_i^t) + \vg_i(\vz_i^t)^\top(\vx - \vz_i^t),
\end{align}   
where $\vz_i^t\in\BR^d$ is some point related to component $f_i$ at the $t$-th iteration.
The estimation (\ref{eq:approx-fi}) motivates us to construct the surrogate problem for the nonlinear least squares (\ref{eq:min_prob}) as follows
\begin{align}\label{eq:prob-ign}
    \min_{\vx\in\BR^d} \psi(\vx)\triangleq \sum_{i=1}^n \psi_i(\vx),
    \qquad\text{where}~\psi_i(\vx)\triangleq\frac{1}{2}\Norm{f_i(\vz_i^t) + \vg_i(\vz_i^t)^\top (\vx - \vz_i^t)}^2.
\end{align}
Since each $\psi_i$ is convex, which implies problem (\ref{eq:prob-ign}) has the closed-form solution
\begin{align}\label{eq:ign-update-0}
\vx^{t+1} 
= \left(\sum_{i=1}^n \vg_i(\vz_i^t)\vg_i(\vz_i^t)^\top\right)^{-1}\sum_{i=1}^n \left(\vg_i(\vz_i^t)^\top\vz_i^t -  f_i(\vz_i^t)\right)\vg_i(\vz_i^t).
\end{align}
We assume the matrix $\sum_{i=1}^n \vg_i(\vz_i^t)\vg_i(\vz_i^t)^\top$ is always non-singular in this subsection, which will be verified under our assumptions in later analysis.

We propose the Incremental Gauss-Newton (\IGN) method by performing an update (\ref{eq:ign-update-0}) at the $t$-th iteration. 
It is worth noting that we can take advantage of the inherent finite-sum structure in formulation (\ref{eq:min_prob}) to establish incremental methods.
Specifically, we update one of $\{\vz^t_i\}_{i=1}^n$ at each iteration in a cyclic fashion, that is
\begin{align}
\label{eq:update-1}
    \vz_i^{t+1} = \begin{cases}
        \vx^{t+1}, & \text{if~} i=i_t, \\
        \vz_i^t, & \text{otherwise},
    \end{cases}   
\end{align}
where $i_t ={t\%n} + 1$. This indicates that we only need to address the terms associated with point~$\vz_{i_t}^t$ in update~(\ref{eq:ign-update-0}), which can be implemented by introducing the aggregated variables
\begin{align}\label{eq:note-1}
\vu^t = \sum_{i=1}^n \left(\vg_i(\vz_i^t)^\top\vz_i^t -  f_i(\vz_i^t)\right)\vg_i(\vz_i^t), \qquad  
\mH^t=\sum_{i=1}^n \vg_i(\vz_i^t)\vg_i(\vz_i^t)^\top
\qquad\text{and}\qquad \mG^t = \left(\mH^t\right)^{-1}.
\end{align}
Then we can write update (\ref{eq:ign-update-0}) as
\begin{align}\label{eq:prob-ign2}
    \vx^{t+1} = \mG^t\vu_t
\end{align}
and maintain the aggregated variables by following recursions\footnote{Noticing that there is no need to explicitly construct matrix $\mH_t$ in implementation, while this matrix is useful to understand and analyze our method.}
\begin{align}\label{eq:recursion-1}
\small
\begin{cases}    
\,\vu^{t+1} = \vu^t - \left(\vg_{i_t}(\vz_{i_t}^t)^\top\vz_{i_t}^{t} - f_{i_t}(\vz_{i_t}^t)\right)\vg_{i_t}(\vz_{i_t}^t) 
 + \left(\vg_{i_t}(\vx^{t+1})^\top\vx^{t+1} - f_{i_t}(\vx^{t+1})\right)\vg_{i_t}(\vx^{t+1}),  \\[0.15cm]
\mH^{t+1} =\mH^t - \vg_{i_t}(\vz_{i_t}^t)\vg_{i_t}(\vz_{i_t}^t)^\top+ \vg_{i_t}(\vx^{t+1})\vg_{i_t}(\vx^{t+1})^\top,  \\[0.15cm]
\mG^{t+1} =\mG^{t} - \mG^{t}\mU^t(\mI + (\mV^t)^\top\mG^{t}\mU^t)^{-1}(\mV^t)^\top\mG^{t}, 
\end{cases}
\end{align}
where the last one is based on Sherman--Morrison--Woodbury formula \cite{woodbury1950inverting} and definitions
\begin{align}
\label{eq:UV-1}
\mU^t \triangleq \big[
    - \vg_{i_t}(\vz_{i_t}^t),~~ 
    \vg_{i_t}(\vx^{t+1})
\big]\in\BR^{d\times 2} 
\qquad\text{and}\qquad
\mV^t \triangleq \big[\vg_{i_t}(\vz_{i_t}^t),~~ 
    \vg_{i_t}(\vx^{t+1})\big]\in\BR^{d\times 2}.
\end{align}
Since each of matrices $\mU^t$ and $\mV^t$ only contains two columns, updating the variables $\vx^{t+1}$, $\vu^{t+1}$ and~$\mG^{t+1}$ can be implemented within the complexity of $\fO(d^2)$ flops. 
Additionally, the memory cost for maintaining variables $\{\vz_i^{t}\}_{i=1}^n$, $\{\vg_{i_t}(\vz^{t})\}_{i=1}^n$, $\vu^t$ and $\mG^t$ is $\fO(nd+d^2)$.
As a comparison, the vanilla Gauss--Newton (\texttt{GN}) method \cite{ben1966newton,nocedal1999numerical}  performs the iteration
\begin{align}\label{eq:classical-gn0}
\begin{split}    
    \vx^{t+1} =& \vx^t - \left(\mJ(\vx^t)^\top\mJ(\vx^t)\right)^{-1}\mJ(\vx^t)^\top\vf(\vx^t),
\end{split}
\end{align}
which takes a computation cost of $\fO(nd+d^3)$ flops and a memory cost of $\fO(nd+d^2)$.

We summarize the procedure of our \IGN~in Algorithm \ref{alg:IGN-1}.
Observe that the vanilla \texttt{GN} iteration (\ref{eq:classical-gn0}) can be reformulated by
\begin{align}\label{eq:classical-gn}
\vx^{t+1} 
= \left(\mJ(\vx^t)^\top\mJ(\vx^t)\right)^{-1}\mJ(\vx^t)^\top(\mJ(\vx^t)\vx^t-\vf(\vx^t)).
\end{align}
Comparing our updates (\ref{eq:prob-ign})--(\ref{eq:prob-ign2}) with (\ref{eq:classical-gn}), the aggregated variables~$\vu^t$, $\mH^t$ and $\mG^t$ can be regarded as the estimators of terms $\mJ(\vx^t)^\top(\mJ(\vx^t)\vx^t-\vf(\vx^t))$, $\mJ(\vx^t)^\top \mJ(\vx^t)$ and $(\mJ(\vx^t)^\top \mJ(\vx^t))^{-1}$ in scheme of~(\ref{eq:classical-gn}) respectively.
The efficiency of our \IGN~method comes from the strategy that we apply the different $\vz_i^t$ in the linear approximation (\ref{eq:approx-fi}) for the different component $f_i$.
In contrast, the vanilla \texttt{GN} method is based on the linear approximation at the identical point $\vx_t$ for all components.

\subsection{The Convergence Analysis}\label{sec:conv-an}

In this subsection, we establish the local superlinear convergence of the proposed \IGN~method. 

We start our analysis from the following proposition, which shows the non-singularity of the Gram matrix associated with the exact Jacobian at the non-degenerate solution $\vx^*\in\BR^d$.

\begin{prop}\label{prop:b-JJ-sing}
Under Assumption \ref{asm:b-J-sing}, it holds that
\begin{align}\label{eq::b-JJ-sing}
    \sigma_{\min} (\mJ(\vx^*)^\top \mJ(\vx^*)) = \mu^2 > 0.
\end{align}
\end{prop}

Under the continuous assumptions on $\vf(\cdot)$ and $\mJ(\cdot)$, we can provide the H\"older continuity of the Gram matrices.

\begin{lem}
\label{le:L-JJ}
{
    Under Assumptions \ref{asm:L-f} and \ref{asm:holder-g}, we have 
    \begin{align*}
    \Norm{\mJ(\vy)^\top\mJ(\vy) - \mJ(\vx)^\top\mJ(\vx)} & \leq 2L_f\fH_\nu\Norm{\vy-\vx}^\nu
    \end{align*}
    and
    \begin{align*}
    \Norm{\vg_i(\vy)\vg_i(\vy)^\top - \vg_i(\vx)\vg_i(\vx)^\top}   \leq 2L_f\fH_\nu\Norm{\vy-\vx}^\nu
    \end{align*}
    for all $\vx, \vy\in\BR^n$ and $i \in [n]$.
}
\end{lem}

Recall the design of \IGN~method is motivated by the estimation $\mH_t\approx\mJ(\vx_t)^\top\mJ(\vx_t)$, which indicates we can connect Proposition \ref{prop:b-JJ-sing} and Lemma \ref{le:L-JJ} to bound the spectrum of $\mH_t$ as follows.

\begin{lem}
\label{le:lb-Ht-new-1}
Under Assumptions \ref{asm:L-f}, \ref{asm:holder-g} and \ref{asm:b-J-sing}, running \IGN~(Algorithm \ref{alg:IGN-1}) with $\mH^0 = \mJ(\vx^0)^\top\mJ(\vx^0)$ and~$\mG^0=(\mH^0)^{-1}$ holds that
\begin{align*}
    \sigma_{\min}(\mH^t) \geq \mu^2 - 2L_f\fH_\nu\sum_{i=1}^n \Norm{\vz_i^t - \vx^*}^\nu
\end{align*}
for all $t\geq 0$.
\end{lem}

Lemma \ref{le:lb-Ht-new-1} indicates that if all of the points $\vz_1^t,\dots,\vz_n^t$ are sufficiently close to the solution $\vx^*$, the matrix $\mH^{t}$ is positive-definite, which guarantees that the inverse of $\mH^{t+1}$ (i.e., matrix $\mG^{t+1}$) in the algorithm is always well-defined.
Based on this intuition, we use induction to show the positive-definiteness of matrices $\mH^t$ and $\mI + (\mV^t)^\top\mG^{t}\mU^t$, and the local superlinear convergence rate of the proposed method.

\begin{thm}
\label{thm:holder-IGN-1}
{
    Under Assumptions \ref{asm:L-f}, \ref{asm:holder-g} and \ref{asm:b-J-sing}, running \IGN~(Algorithm \ref{alg:IGN-1}) with initialization $\vx^0\in\BR^d$, $\mH^0 = \mJ(\vx^0)^\top\mJ(\vx^0)$ and $\mG^0=(\mH^0)^{-1}$ such that
    \begin{align*}
    \Norm{\vx^0 - \vx^*}\leq \left(\frac{\mu^2}{4L_f\fH_\nu n}\right)^{1/\nu},
    \end{align*} 
    we have $\mH^t \succeq (\mu^2/2)\mI$ and $\sigma_{\min}(\mI + (\mV^t)^\top\mG^{t}\mU^t)>0$ for all $t\geq 0$.
    Additionally, there exists sequence $\{r^t\}$ such that $\Norm{\vx^t-\vx^*}\leq r_t$ and it holds

    \begin{align*}
     r_{t+1}\leq  c^{(1+\nu)^{\left(\floor{{t}/{n}}-1\right)}}r_t \qquad\text{with}\qquad c = 1- \frac{1}{n}\left(1-\left(\frac{1}{2 
  (1+\nu)}\right)^{(1+\nu)}\right)
    \end{align*}
    for all $t\geq n$.
}
\end{thm}

Observe that the term of $c$ in Theorem \ref{thm:holder-IGN-1} is monotonically decreasing with respect to $\nu\in(0,1]$, we can bound it by $1-15/(16n) \leq c < 1-1/(2n)$ and simplify the superlinear convergence as follows.
\begin{cor}
\label{cor:ign1}
    Under the settings and notations of Theorem \ref{thm:holder-IGN-1}, we have
    \begin{align*}
        r_{t+1} < \Big(1-\frac{1}{2n}\Big)^{(1+\nu)^{\left(\floor{{t}/{n}}-1\right)}}r_t
    \end{align*}
    for all $t\geq n$.
\end{cor}

Theorem \ref{thm:holder-IGN-1} also indicates that the larger $\nu\in(0,1]$ leads to faster superlinear convergence rate.
In the case of $\nu=1$, our H\"older continuous condition (Assumption \ref{asm:holder-g}) degenerates to the Lipschitz continuity, then we can achieve the $n$-step local quadratic convergence rate as follows.

\begin{cor}
\label{cor:qua-1}
Under the settings and notations of Theorem \ref{thm:holder-IGN-1} with $\nu=1$, we have the $n$-step quadratic convergence
\begin{align*}
    r_{t} \leq \frac{1}{4}r_{t-n}^{2}
\end{align*}
for all $t\geq n$.
\end{cor}

\section{The Extension to Mini-Batch Methods}\label{sec:ign_k}
\label{sec:group}

We can also improve the efficiency of \IGN~method by using the mini-batch update. 
Specifically, we consider the mini-batch size $k$ and  divide the indices into $m=\ceil{n/k}$ non-overlapping subsets,  
i.e., we partition the index set $[n]=\{1,\dots,n\}$ into subsets $\{\fS_1,\dots,\fS_m\}$ such that $|\fS_1|=\dots=|\fS_{m-1}|= k$, $\cup_{i=1}^m\fS_i=[n]$ and~$\fS_i\cap\fS_j=\emptyset$ for all distinct~$i,j\in[k]$. 

The mini-batch variant of \IGN~also apply the update of the form $\vx^{t+1}=\mG^t\vu^t$.
Different from \IGN, we update variables $\{\vz_i^t\}_{i=1}^m$ with the smaller period~$m=\ceil{n/k}$ such that
\begin{align}
\label{eq:update-2}
    \vz_i^{t+1} = \begin{cases}
        \vx^{t+1}, & \text{if~} i=i_t, \\
        \vz_i^t, & \text{otherwise},
    \end{cases}   
\end{align}
where $i_t ={t\%m} + 1$.

We establish recursions of aggregated variables by a mini-batch way as follows
\begin{align}\label{eq:recursion-k}
\small
\!\!\!\!\begin{cases}    
\displaystyle{\vu^{t+1} \!=\! \vu^t \!-\!\! \sum_{j\in\fS_{i_t}}\!\!\left(\vg_{j}(\vz_{i_t}^t)^\top\vz_{i_t}^{t} \!-\! f_{j}(\vz_{i_t}^t)\right)\vg_{j}(\vz_{i_t}^t) 
 \!+\!\!\sum_{j\in\fS_{i_t}}\!\!\left(\vg_{j}(\vx^{t+1})^\top\vx^{t+1} \!-\! f_{j}(\vx^{t+1})\right)\vg_{j}(\vx^{t+1}),}  \\[0.4cm]
\displaystyle{\mH^{t+1} \!=\! \mH^t - \sum_{j\in\fS_{i_t}}\vg_{j}(\vz_{i_t}^t)\vg_{j}(\vz_{i_t}^t)^\top+ \sum_{j\in\fS_{i_t}}\vg_{j}(\vx^{t+1})\vg_{j}(\vx^{t+1})^\top,}  \\[0.4cm]
\displaystyle{\mG^{t+1} \!=\! \mG^{t} - \mG^{t}\mU^t(\mI + (\mV^t)^\top\mG^{t}\mU^t)^{-1}(\mV^t)^\top\mG^{t},} 
\end{cases}
\end{align}
where we construct matrices $\mU^t, \mV^t\in\BR^{d\times 2|\fS_{i_t}|}$ as
\begin{align*}
\begin{cases}    
\mU^t = \Big[- \vg_{j_1}(\vz_{i_t}^t),~~ 
    \vg_{j_1}(\vx^{t+1}),~\cdots~,~
    - \vg_{j_{|\fS_{i_t}|}}(\vz_{i_t}^t),~~
    \vg_{j_{|\fS_{i_t}|}}(\vx^{t+1})
\Big],\\[0.2cm]
\mV^t = \Big[\vg_{j_1}(\vz_{i_t}^t),~~ 
    \vg_{j_1}(\vx^{t+1}),~\cdots~,~
    \vg_{j_{|\fS_{i_t}|}}(\vz_{i_t}^t),~~
    \vg_{j_{|\fS_{i_t}|}}(\vx^{t+1})
\Big],
\end{cases}
\end{align*}
and indices $j_1,\dots,j_{|\fS_{i_t}|}$ are the elements in subset $\fS_{i_t}$ such that $|\fS_{i_t}|\leq k$.

We formally present the procedure of the Mini-Batch Incremental Gauss-Newton (\MBIGN) method~in Algorithm~\ref{alg:MB-IGN} (see Appendix \ref{sec:alg-k}).
The memory cost of \MBIGN~is $\fO(nd+d^2)$, matching the complexity of $\IGN$.
Each iteration of \MBIGN~includes the matrix multiplication of $\mG^t$, $\mU^t$ and $\mV^t$ within the complexity of $\fO(kd^2)$ flops.
It is worth noting that the mini-batch update in \MBIGN~can be efficiently implemented by block matrix operation that takes advantage of parallel
computation \cite{davis1998block}.

Formally, we present the following convergence results of \MBIGN.

\begin{thm}
\label{thm:Group-1}
{
    Under Assumptions \ref{asm:L-f}, \ref{asm:holder-g} and \ref{asm:b-J-sing}, running \MBIGN~(Algorithm \ref{alg:MB-IGN}) with mini-batch size $k$ and initialization $\vx^0\in\BR^d$, $\mH^0 = \mJ(\vx^0)^\top\mJ(\vx^0)$ and $\mG^0=(\mH^0)^{-1}$ such that
    \begin{align*}
    \Norm{\vx^0 - \vx^*}\leq \left(\frac{\mu^2}{4kL_f\fH_\nu\ceil{{n}/{k}}}\right)^{1/\nu}, 
    \end{align*} 
    we have $\mH^t \succeq (\mu^2/2)\mI\,$ and $\sigma_{\min}(\mI + (\mV^t)^\top\mG^{t}\mU^t)>0$ for all $t\geq 0$.
    Additionally, there exists sequence $\{r_t\}$ such that $\Norm{\vx^t-\vx^*}\leq r_t$ and it holds

    \begin{align*}
     r_{t+1}\leq  c^{(1+\nu)^{\left(\floor{\frac{t}{\ceil{n/k}}}-1\right)}}r_t \qquad\text{with}\qquad c = 1- \frac{1}{\ceil{n/k}}\left(1-\left(\frac{1}{2 
  (1+\nu)}\right)^{(1+\nu)}\right).
    \end{align*}
}

\end{thm}
The terms of $n/k$ in the results of Theorem \ref{thm:Group-1} imply that increasing mini-batch size $k$ can speed up the convergence of $\MBIGN$. 
Additionally, the convergence of $\MBIGN$ matches $\IGN$ if we take $k=1$.

Similar to the discussion in Section \ref{sec:conv-an}, we have the following corollary for \MBIGN~method. 
\begin{cor}
\label{cor:ignk}
    Under settings of Theorem \ref{thm:Group-1}, we have
    \begin{align*}
        r_{t+1}\leq \Big(1-\frac{1}{2\ceil{n/k}}\Big)^{(1+\nu)^{\left(\floor{\frac{t}{\ceil{n/k}}}-1\right)}}r_t
    \end{align*}
    for all $t \geq \ceil{n/k}$. In the case of $\nu=1$, we have the $\ceil{n/k}$-step quadratic convergence
    \begin{align*}
        r_{t} \leq \frac{1}{4} r_{t-\ceil{n/k}}^2
    \end{align*}
    for all $t \geq \ceil{n/k}$.
\end{cor}

Specifically, Corollary \ref{cor:ignk} indicates that the \MBIGN~ method with $k=n$ has the quadratic convergence under the assumption of Lipschitz continuous Jacobian (Assumption \ref{asm:holder-g} with $\nu=1$), which matches the rate of vanilla Gauss--Newton method.

\section{Related Work}\label{sec:related_work}

\begin{table*}[t]
\caption{We compare the per-iteration computation complexity, memory cost, convergence rates and the assumption of Jacobin of proposed methods and baselines. The rightmost column means that the methods \texttt{GN}, \texttt{SNR},  \texttt{GN-BFGS}, \texttt{BFB} and \texttt{BBB} require to access all of the components $f_1,\dots,f_n$ at each iteration, while the other methods only require to access one or mini-batch of components.} \label{tab:algs}

\vskip -0.1cm

\resizebox{\linewidth}{!}{
\begin{threeparttable}
\footnotesize\setlength\tabcolsep{5.pt}
\begin{tabular}{ccccccc}
\toprule
Methods & Computation & Memory & Convergence & Jacobian & $f_i$ \\
\midrule

\texttt{GN} \cite{ben1966newton,nocedal1999numerical} & $\fO(nd^2 + d^3)$ & $\fO(nd + d^2)$ & quadratic & Lipschitz & {\XSolidBrush} \\[0.05cm]
\texttt{SNR} \cite{yuan2022sketched}\tnote{$\sharp$} & $\fO(n\tau^2 + \tau^3)$  & $\fO(\tau d)$ & sublinear 
& Lipschitz & {\XSolidBrush} \\[0.05cm]

\texttt{GN-BFGS} \cite{li1999globally}\tnote{$\ddag$} & $\fO(d^2)$ & $\fO(d^2)$ & asymptotic superlinear & H\"older & {\XSolidBrush} \\[0.05cm]
\texttt{BGB} \cite{liu2023block}\tnote{$\S$} & $\fO(\tilde kd^2)$ & $\fO(d^2)$ & $\fO\big((1 - {\tilde k}/d)^{t(t-1)/4}\big)$ & Lipschitz & {\XSolidBrush} \\[0.05cm]
\texttt{BBB} \cite{liu2023block}\tnote{$\S$} & $\fO(\tilde kd^2)$ & $\fO(d^2)$ & $\fO\big((1 - \tilde k/(\varkappa d))^{t(t-1)/4}\big)$ & Lipschitz & {\XSolidBrush} \\[0.05cm]
\texttt{EKF} \cite{bertsekas1996incremental} & $\fO(d^2)$ & $\fO(d^2)$ &  sublinear & Lipschitz & {\Checkmark} \\[0.05cm]
\texttt{EKF-S} \cite{moriyama2003incremental, gurbuzbalaban2015globally} & $\fO(d^2)$ & $\fO(d^2)$ & linear & Lipschitz & {\Checkmark} \\[0.05cm]
\texttt{IGN} (this work) & $\fO(d^2)$ & $\fO(nd+d^2)$ & $\fO\big((1 - 1/(2n))^{(1+\nu)^{\floor{t/n}}}\big)$ & H\"older & {\Checkmark} \\[0.05cm]
\texttt{MB-IGN} (this work) & $\fO(kd^2)$  & $\fO(nd+d^2)$ & $\fO\big((1 -  k/(2n))^{(1+\nu)^{\floor{{k t}/{n}}}}\big)$ & H\"older & {\Checkmark} \\ 
\bottomrule    
\end{tabular}
\begin{tablenotes}
{\scriptsize 
\item [{$\sharp$}] The \texttt{SNR} method requires the star convexity in their minimization formulation. The notation $\tau$ presents the sketch size. \\
\item [{$\ddag$}] The \texttt{GN-BFGS} method requires $n=d$ and the Jacobian is symmetric.\\
\item [{$\S$}] The \texttt{BGB} and \texttt{BBB} methods requires $n=d$. The notation $\tilde k$ is rank of the modification matrix and \mbox{$\varkappa\triangleq L_f/\mu$} is the condition number.} 
\end{tablenotes}  
\end{threeparttable}}	

\end{table*}

We compare the theoretical results of proposed \IGN~and \MBIGN~with existing methods in Table~\ref{tab:algs}.

The methods including Gauss--Newton-based BFGS (\texttt{GN-BFGS}) \cite{li1999globally}, Block Good Broyden's method (\texttt{BGB}) \cite{liu2023block}, Block Bad Broyden's method (\texttt{BBB}) \cite{liu2023block} and Sketched Newton--Raphson (\texttt{SNR}) \cite{yuan2022sketched} only focus on establishing the Jacobian estimator, while each of their iteration depends on accessing all components in the nonlinear vector function that is expensive for large-scale problems.
In addition, the quasi-Newton methods including \texttt{GN-BFGS}~\cite{li1999globally}, \texttt{BGB}~\cite{liu2023block} and \texttt{BBB}~\cite{liu2023block} only work for the scenario of~$n=d$.
The \texttt{SNR} method enjoys an efficient update for large $n$, while it lacks the local superlinear convergence like classical Newton-type methods.

The Extended Kalman Filter with Stepsize (\texttt{EKF-S}) \cite{moriyama2003incremental, gurbuzbalaban2015globally} is based on the incremental update that only accesses one (or mini-batch) of components and the corresponding gradient at each iteration.
Concretely, the \texttt{EKF-S} method performs the iteration
\begin{align*}
    \vx^{t+1} = \vx^t - \alpha^t (\tilde\mH^{t})^{-1}\vg_{i_t}(\vx^{t})f_{i_t}(\vx^t)
\end{align*}
with some stepsize $\alpha^t>0$, where $\tilde\mH^t\in\BR^{d\times d}$ is the estimator for the Gram matrix $\mJ(\vx^t)^\top\mJ(\vx^t)$ which is constructed by the recursion
\begin{align}\label{eq:update-ekf-H}
    \tilde\mH^{t+1} = \lambda^t\tilde\mH^{t} + \vg_{i_t}(\vx^{t+1})\vg_{i_t}(\vx^{t+1})^\top
\end{align}
for some $\lambda^t\in(0,1]$. 
The original Extended Kalman Filter method (\texttt{EKF}) \cite{bertsekas1996incremental} takes a fixed stepsize of~$\alpha^t=1$ in the above iteration and achieves a sublinear convergence rate.
Later, \citet{gurbuzbalaban2015globally} showed that introducing the adaptive stepsize can achieve the linear convergence rate.
Note that \texttt{EKF-S} and \texttt{EKF} will not explicitly reuse the information of vector $g_{i_t}(\vx_t)$ in later iterations. 
In other words, the recursion~(\ref{eq:update-ekf-H}) indicates all information of the historical gradient is heuristically compressed into the term of $\lambda^t\tilde\mH^t$.
In contrast, the proposed \IGN~method establishes the Gram matrix approximation~$\mH_t\approx\mJ(\vx^t)^\top\mJ(\vx^t)$ by equations~(\ref{eq:note-1}) and~(\ref{eq:recursion-1}), 
which clearly corresponds to the linear approximation (\ref{eq:approx-fi})-(\ref{eq:prob-ign}) by reusing all of the historical gradients $\{\vg_{i}(\vz_i^t)\}_{i=1}^n$.
This strategy encourages a more accurate Gram matrix estimation in our method and leads to a superlinear convergence rate. 

The incremental Newton-type methods have also been studied in finite-sum strongly convex optimization~\cite{rodomanovnewton, rodomanov2016superlinearly,mokhtari2018iqn,lahoti2023sharpened,liu2024incremental}.
In the view of our formulation (\ref{eq:ori_prob}), this work considers solving the system of nonlinear equations of the form $\vf(\vx)=\vzero$, where $\vf:\BR^d\to\BR^d$ is the gradient of some objective function and has the finite-sum structure $\vf(\vx)\triangleq(1/N)\sum_{i=1}^N\vf_i(\vx)$ with symmetric positive-definite Jacobian.
These methods can achieve superlinear convergence rates by accessing one of~$\{\vf_i\}_{i=1}^N$ and its Jacobian at each iteration.
However, their iterations have to maintain Jacobians for all of the individuals $\{\vf_i\}_{i=1}^N$ with a memory cost of $\fO(Nd^2)$, which is prohibitive for a large $N$.

\section{Experiments}\label{sec:experiments}

We conduct numerical experiments on the following applications:
\begin{itemize}[leftmargin=0.5cm,topsep=-0.03cm,itemsep=-0.1cm]
\item Regularized Logistic Regression: We consider training the binary classifier~$\vx\in\BR^d$ by solving the nonconvex regularized logistic regression problem  \cite{antoniadis2011penalized,kohler2017sub}
\begin{align*}
    \min_{\vx\in\BR^d} \ell(\vx) \triangleq \frac{1}{N}\sum_{j=1}^N\log(1 + \exp(-b_j\va_j^\top \vx)) + \theta\sum_{k=1}^d \frac{\nu x_k^2}{1+\nu x_k^2},
\end{align*}
where $\{(\va_j, b_j)\}_{j=1}^{N}$ is the training set such that $\va_j\in\BR^d$ and $b_j\in\{-1, 1\}$ for all $j\in[N]$. 
We set $\theta=10^{-2}$ and $\nu=1$ for the model. 
We formulated the above minimization problem by the formulation of nonlinear equations (\ref{eq:ori_prob}) with $ \vf(\vx)\triangleq \nabla \ell(\vx)$.
We perform the experiments on dataset ``DBWorld'' ($N=64$ and $d=4,702$) \cite{misc_dbworld_e-mails_219} for this problem.
\item Chandrasekhar's H-Equation: 
We consider the Chandrasekhar's H-equation, which is widely used in analytical radiative transfer theory~\cite{hottel1967radiative,chandrasekhar1960radiative}. 
It can be formulated by problem (\ref{eq:ori_prob}) with
\begin{align*}
    f_i(\vx) = x_i - \left(1 - \frac{c}{2n}\sum_{j=1}^n\frac{\mu_i x_j}{\mu_i + \mu_j}\right)^{-1} 
    ~\text{for all}~i\in[n],
    \quad\text{where}\quad
    \mu_i = \frac{i-1/2}{n}.
\end{align*}
We set $d=2,000$ and $c=1-10^{-5}$ for this problem in our experiments. .
\item Soft Maximum Minimization: We consider the soft maximum minimization problem~\cite{nesterov2005smooth,bullins2020highly}
\begin{align}
\label{eq:soft-func}
    \min_{\vx\in\BR^d} h(\vx) \triangleq \mu \ln{\left(\sum_{i=1}^N\exp{\left(\frac{\inner{\va_i}{\vx} - b_i}{\mu}\right)}\right)} + \frac{\lambda}{2}\Norm{x}^2,
\end{align}
which can be formulated by problem (\ref{eq:ori_prob}) with $\vf(\vx)\triangleq\nabla h(\vx)$. 
We follow the setting of \cite{doikov2023second,doikov2024super} by generating the entries of $\va_1, \cdots, \va_N\in\BR^d$ and $\vb\in\BR^N$ randomly and independently from the uniform distribution on $[-1, 1]$. 
We set $N=2000$, $d=2000$, $\mu=5$ and $\lambda=2$ in our experiments.
\end{itemize}

We first investigate the impact of mini-batch size $k$ of \MBIGN~method (Algorithm~\ref{alg:MB-IGN}) on the performance. We run \MBIGN~by taking the different mini-batch sizes on the three problems and present the empirical results for time~(s) against $||\vf(\vx)||$ in Figure \ref{fig:k-compare}, where the setting $k=1$ corresponds to our \IGN~method (Algorithm \ref{alg:IGN-1}).
We can observe that the mini-batch update is effective in reducing the time cost.
The mini-batch sizes of $500$, $200$, and $100$ achieve the best performance on the problems of robust logistic regression, Chandrasekhar’s H-equation, and soft maximum minimization, respectively.

We then compare the proposed methods \MBIGN~(Algorithm \ref{alg:MB-IGN}) with baseline methods \texttt{SNR} \cite{yuan2022sketched}, \texttt{EKF-S} \cite{bertsekas1996incremental, moriyama2003incremental}, \texttt{BGB} \cite{liu2023block} and \texttt{BBB} \cite{liu2023block}. 
We present the empirical results for the number of epochs against $||\vf(\vx)||$ in Figure~\ref{fig:interactions}, where one epoch means one complete pass of all components of the nonlinear vector function.
We can obverse that the proposed \MBIGN~and the baseline method \texttt{BGB} outperforms others on all problems.
This is reasonable since only these two methods enjoy the explicit condition-number-free superlinear convergence rates (see Table \ref{tab:algs}). 
The superlinear convergence rate of \texttt{BBB} method depends on the condition number, which leads to its performance not always better than the linear convergent method $\texttt{EKF-S}$.

We also present the empirical results for the cost of time (second) against $||\vf(\vx)||$ in Figure \ref{fig:t-compare}.
We can obverse that the proposed \MBIGN~always performs significantly better than all baseline methods. 
This is in line with our expectations because only our \MBIGN~method enjoys both the superlinear convergence rate and the cheap iteration cost.
Although the \texttt{BGB} method has a comparable number of epochs to our~\MBIGN~on the problem of solving Chandrasekhar's H-Equation, the iteration with accessing all components makes its time cost expensive.

\begin{figure}[t]
    \centering
    \begin{subfigure}[c]{0.32\textwidth}
        \includegraphics[width=\textwidth]{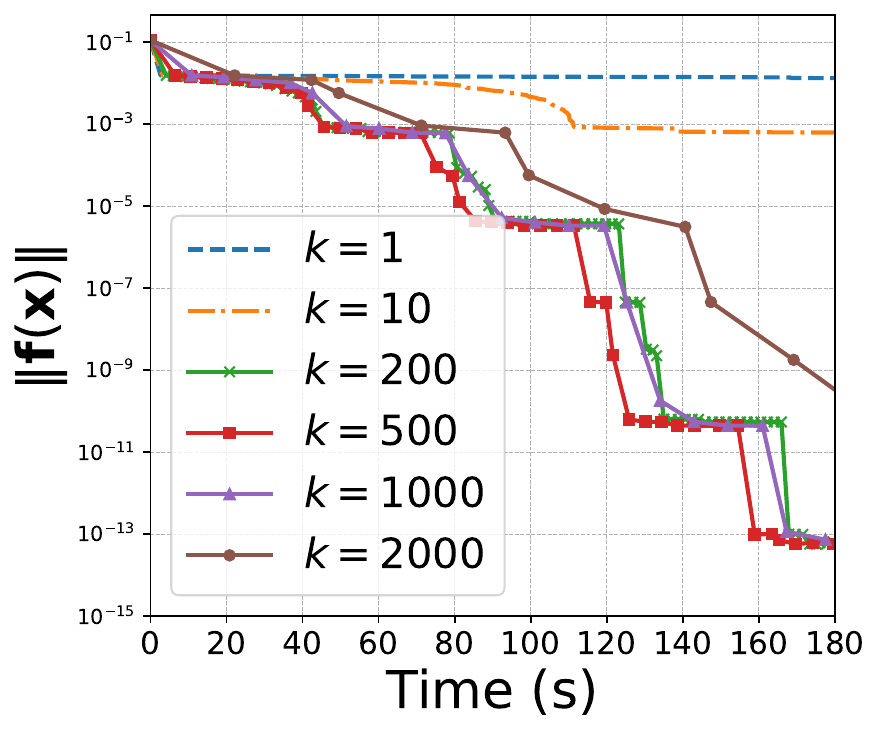}
        \caption{Robust Logistic Regression}
    \end{subfigure}
    \hfill
    \begin{subfigure}[c]{0.32\textwidth}
        \includegraphics[width=\textwidth]{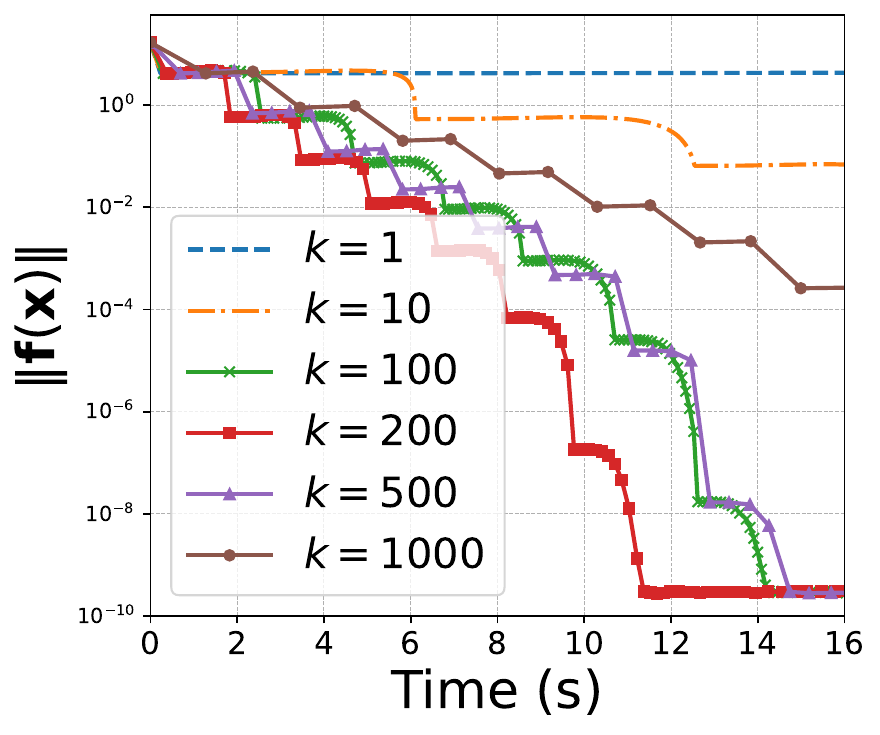}
        \caption{Chandrasekhar's H-Equation}
    \end{subfigure}
    \hfill
    \begin{subfigure}[c]{0.32\textwidth}
        \includegraphics[width=\textwidth]{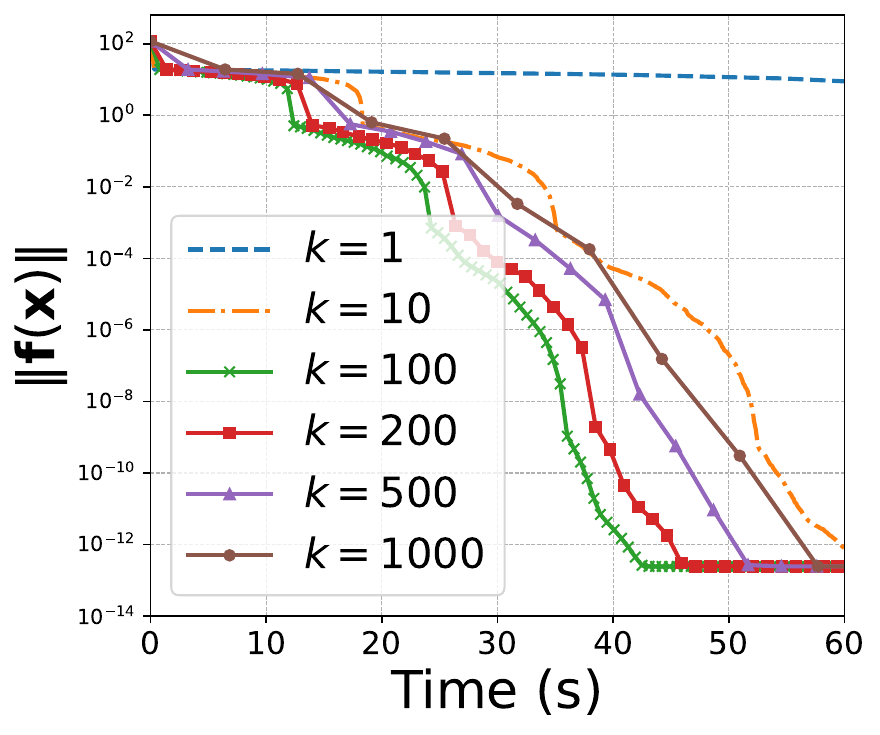}
        \caption{Soft Maximum Minimization}
    \end{subfigure}
    \caption{Experimental results of time (s) vs. $\|\vf(\vx)\|$ for \MBIGN~with different mini-batch size $k$.}
    \label{fig:k-compare}
\end{figure}

\begin{figure}[H]
    \centering
    \begin{subfigure}[c]{0.32\textwidth}
        \includegraphics[width=\textwidth]{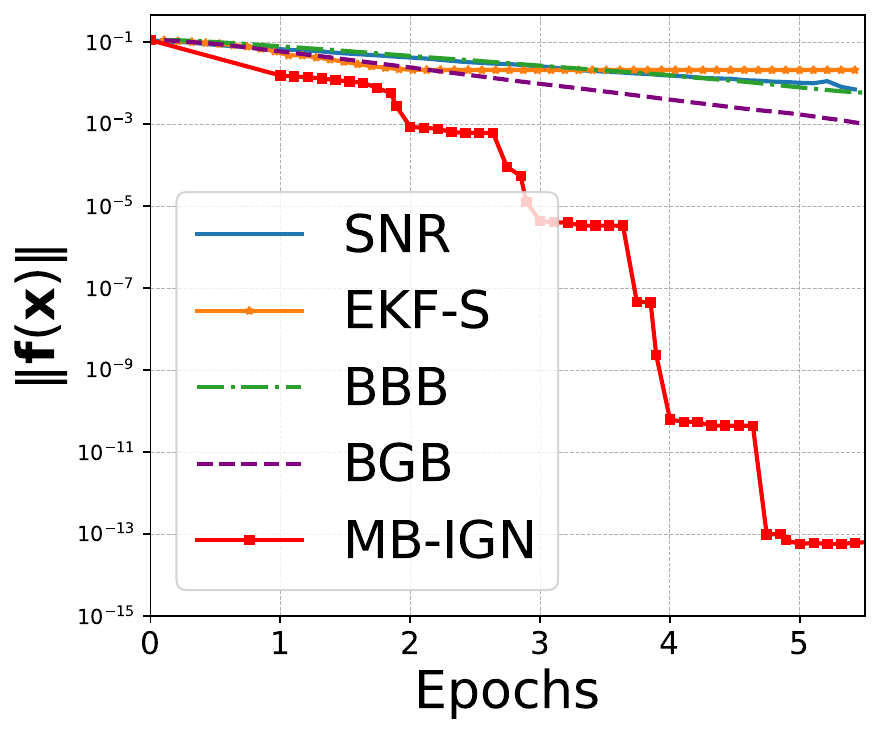}
        \caption{Robust Logistic Regression}
    \end{subfigure}
    \hfill
    \begin{subfigure}[c]{0.32\textwidth}
        \includegraphics[width=\textwidth]{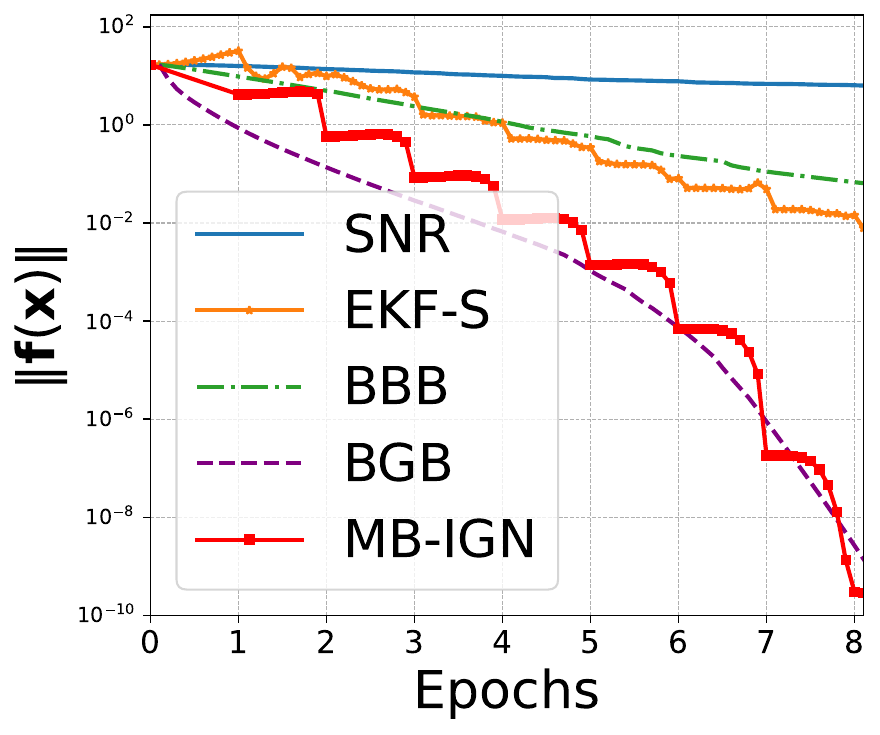}
        \caption{Chandrasekhar's H-Equation}
    \end{subfigure}
    \hfill
    \begin{subfigure}[c]{0.32\textwidth}
        \includegraphics[width=\textwidth]{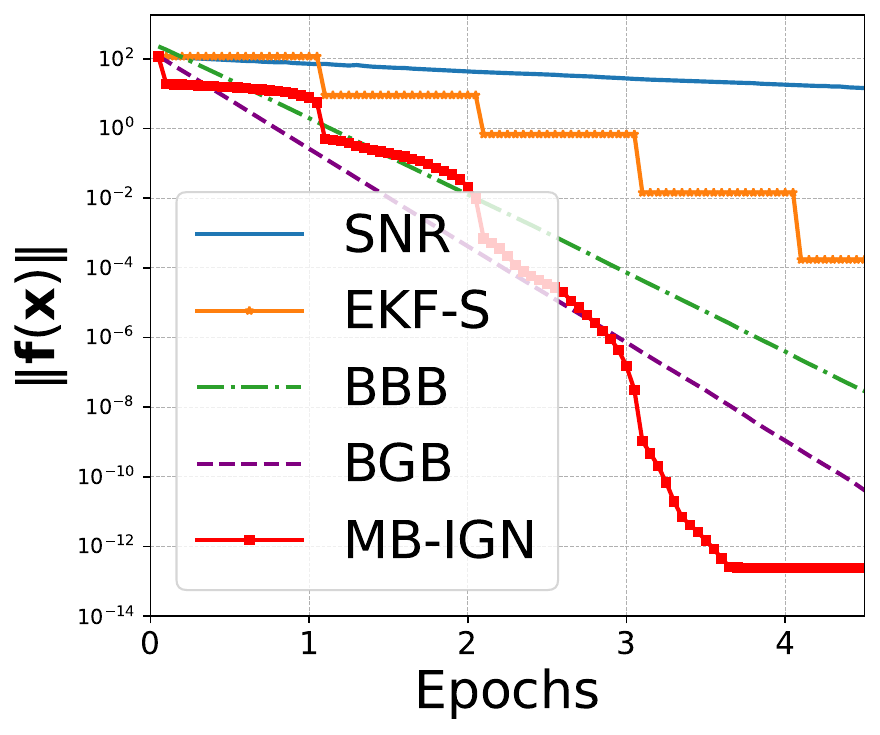}
        \caption{Soft Maximum Minimization}
    \end{subfigure}
    \caption{Experimental results of epochs vs. $\|\vf(\vx)\|$ for all methods.}
    \label{fig:interactions}
\end{figure}

\begin{figure}[H]
    \centering
    \begin{subfigure}[c]{0.32\textwidth}
        \includegraphics[width=\textwidth]{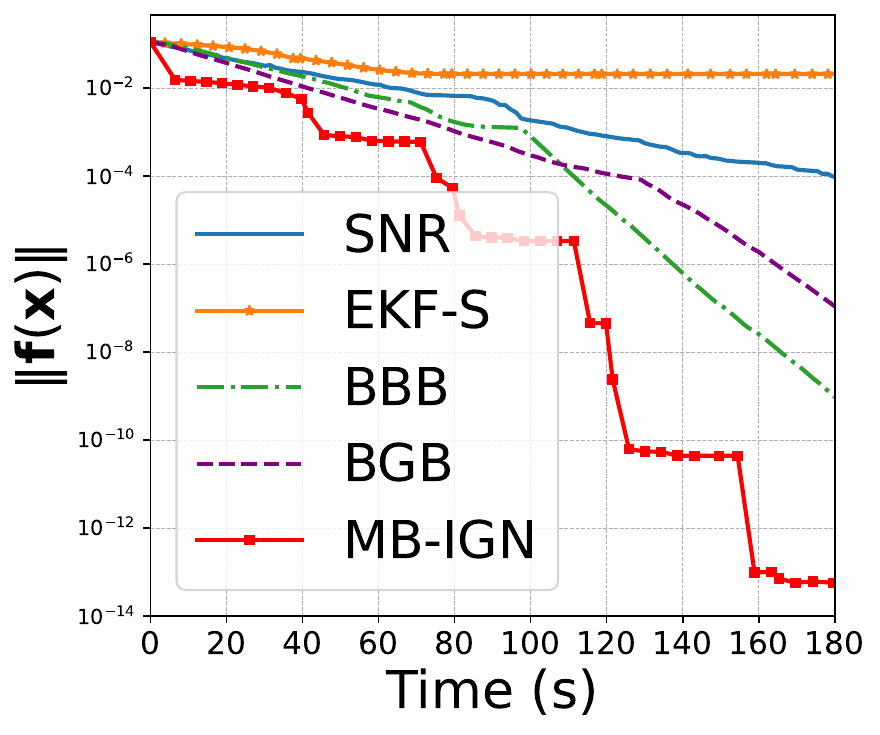}
        \caption{Robust Logistic Regression}

    \end{subfigure}
    \hfill
    \begin{subfigure}[c]{0.32\textwidth}
        \includegraphics[width=\textwidth]{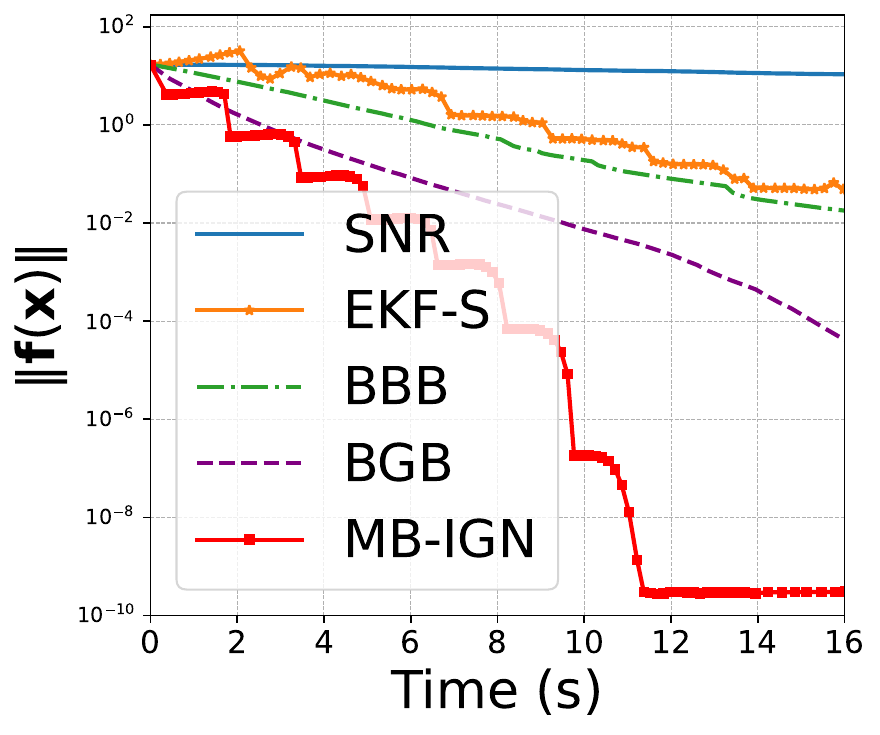}
        \caption{Chandrasekhar's H-Equation}
    \end{subfigure}
    \hfill
    \begin{subfigure}[c]{0.32\textwidth}
        \includegraphics[width=\textwidth]{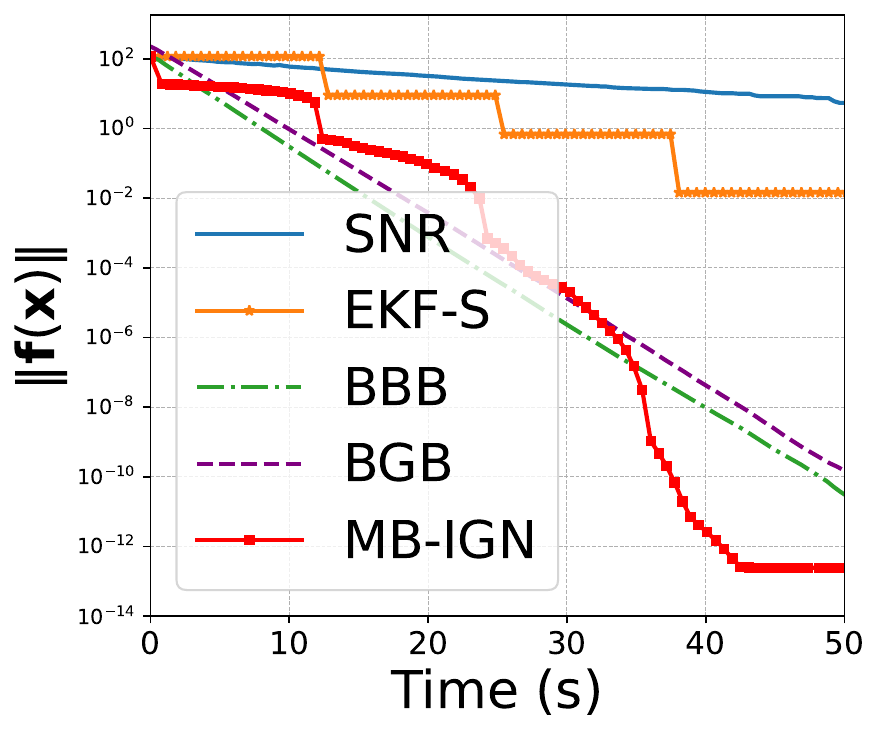}
        \caption{Soft Maximum Minimization}
    \end{subfigure}
    \caption{Experimental results of time (s) vs. $\|\vf(\vx)\|$ for all methods.}
    \label{fig:t-compare}
\end{figure}

\section{Conclusion}\label{sec:conclusion}
In this work, we propose the incremental Gauss--Newton method (\texttt{IGN}) for solving the system of nonlinear equations.  
We design the algorithm by tracking the historical gradient of all components to establish the estimator of the Gram matrix (its inverse).
The theoretical analysis shows \texttt{IGN} enjoys the explicit superlinear convergence rate under the assumption of H\"older continuous Jacobian.
We also provide a mini-batch extension of our \texttt{IGN} method (\MBIGN) and show it has an even faster superlinear convergence rate.
The numerical experiments on the applications of regularized logistic regression, Chandrasekhar's H-equation, and soft maximum minimization validate the advantage of the proposed methods over existing baselines.

In the future, it will be interesting to study the incremental Gauss--Newton method to solve nonlinear equations in the distributed setting. 
It is also possible to design incremental quasi-Newton methods for solving the general nonlinear equations.

\bibliographystyle{plainnat}
\bibliography{reference}

\newpage
\appendix

\setlength{\parindent}{0pt}

{\LARGE \textbf{Appendix}\par}

~\\

The appendix is organized as follows.
In Section \ref{sec:alg-k}, we provide the detailed procedure of Mini-Batch Incremental Gauss–Newton Method (\texttt{MB-IGN}). 
In Section \ref{sec:basic-results}, we provide some results for Jacobians
In Section~\ref{sec:series}, we introduces an auxiliary sequence and analyze its properties.
In Sections \ref{sec:con-ign} and \ref{sec:con-mbign}, we provide the convergence analysis for proposed \texttt{IGN} and \texttt{MB-IGN}, respectively.

\section{The Mini-Batch Incremental Gauss–Newton Method}
\label{sec:alg-k}

We provide the detailed procedure of Mini-Batch Incremental Gauss–Newton Method (\texttt{MB-IGN}) in Algorithm \ref{alg:MB-IGN}.

\begin{algorithm}[ht!]
\caption{Mini-Batch Incremental Gauss–Newton Method (\texttt{MB-IGN})}
\label{alg:MB-IGN}
\begin{algorithmic}[1]
\State \textbf{Input:} $\vx^0\in\BR^d$, $\vu^0\in\BR^d$, $\mH^0,\mG^0\in\BR^{d\times d}$, $k\leq n$, $m=\ceil{n/k}$  \vskip0.1cm

\State Partition the index set $[n]=\{1,\dots,n\}$ into subsets $\{\fS_1,\dots,\fS_m\}$ such that \vskip0.12cm
$|\fS_1|=\dots=|\fS_{m-1}|= k$,~~
$\cup_{i=1}^m\fS_i=[n]$~~and~~
$\fS_i\cap\fS_j=\emptyset$~~for all~$i,j\in[k]$ \vskip0.12cm
\State \textbf{for} $t=0,1,\dots$ \vskip0.12cm
\State \quad  $\vx^{t+1}=\mG^t\vu^t$ \vskip0.12cm
\State \quad  $i_t =t\%m + 1$ \vskip0.12cm

\State \quad 
$\mU^t = \Big[- \vg_{j_1}(\vz_{i_t}^t),~~ 
    \vg_{j_1}(\vx^{t+1}),~\cdots~,~
    - \vg_{j_{|\fS_{i_t}|}}(\vz_{i_t}^t),~~
    \vg_{j_{|\fS_{i_t}|}}(\vx^{t+1})
\Big]$ \vskip0.12cm
\State \quad $\mV^t = \Big[\vg_{j_1}(\vz_{i_t}^t),~~ 
    \vg_{j_1}(\vx^{t+1}),~\cdots~,~
    \vg_{j_{|\fS_{i_t}|}}(\vz_{i_t}^t),~~
    \vg_{j_{|\fS_{i_t}|}}(\vx^{t+1})
\Big]$ \vskip0.12cm

 \State \quad $\displaystyle{\vu^{t+1} \!=\! \vu^t \!-\!\! \sum_{j\in\fS_{i_t}}\!\!\left(\vg_{j}(\vz_{i_t}^t)^\top\vz_{i_t}^{t} \!-\! f_{j}(\vz_{i_t}^t)\right)\vg_{j}(\vz_{i_t}^t) 
 \!+\!\!\sum_{j\in\fS_{i_t}}\!\!\left(\vg_{j}(\vx^{t+1})^\top\vx^{t+1} \!-\! f_{j}(\vx^{t+1})\right)\vg_{j}(\vx^{t+1})}$\vskip0.12cm
\State \quad $\displaystyle{\mH^{t+1} \!=\! \mH^t - \sum_{j\in\fS_{i_t}}\vg_{j}(\vz_{i_t}^t)\vg_{j}(\vz_{i_t}^t)^\top+ \sum_{j\in\fS_{i_t}}\vg_{j}(\vx^{t+1})\vg_{j}(\vx^{t+1})^\top}$ \vskip0.12cm
\State\label{line:update-G-k} \quad  $\mG^{t+1} =\mG^{t} - \mG^{t}\mU^t(\mI + (\mV^t)^\top\mG^{t}\mU^t)^{-1}(\mV^t)^\top\mG^{t}$ \vskip0.12cm

\State \quad  $\vz_i^{t+1} = \begin{cases}
        \vx^{t+1}, & \text{if~} i=i_t \\
        \vz_i^t, & \text{otherwise}
    \end{cases}$ \vskip0.12cm
\State \textbf{end for}
\end{algorithmic}
\end{algorithm}

\section{Some Basic Results for Jacobians}
\label{sec:basic-results}

This section presents some useful results for our later analysis.

\begin{lem}
    \label{le:H-con}
    \textnormal{(H\"older continuity of each gradient)} Under Assumption \ref{asm:holder-g}, it satisfies that
    \begin{align}
        \Norm{\vg_i(\vy)-\vg_i(\vx)}\leq \fH_\nu\Norm{\vy-\vx}^\nu,
    \end{align}
    for any $\vx, \vy\in\BR^d$, and $i\in[n]$.
\end{lem}

\begin{proof}
    We denote
    \begin{align*}
    \Tilde{\mJ} = \begin{bmatrix}
        (\vg_1(\vy)-\vg_1(\vx))^\top \\
        \vdots \\
        (\vg_n(\vy)-\vg_n(\vx))^\top 
    \end{bmatrix}\in\BR^{n\times d} \qquad\text{with}\qquad \vg_i(\vx)=\nabla f_i(\vx)
    \end{align*}
    and let $\ve_i\in\BR^{n}$ be the $i$-th standard basic vector in $n$-dimensional Euclidean space. 
    Then the facts~$\Tilde{\mJ}=\mJ(\vy)-\mJ(\vx)$ and $\Tilde{\mJ}^\top \ve_i=\vg_i(\vy)-\vg_i(\vx)$ imply we have
    \begin{align*}
    \Norm{\vg_i(\vy)-\vg_i(\vx)}
    \leq \|\Tilde{\mJ}^\top\|\Norm{\ve_i} = 
    \|\Tilde{\mJ}\|=\Norm{\mJ(\vy)-\mJ(\vx)}\leq \fH_\nu\Norm{\vy-\vx}^\nu,
    \end{align*}
    where the last step is based on the H\"older continuouity of $\mJ(\cdot)$.
\end{proof}

\begin{lem}
\label{le:power-bound}
\textnormal{(Bound for H\"older-continuous function)} Under Assumption \ref{asm:holder-g}, we have
    \begin{align}
        f_i(\vy) - f_i(\vx) - \vg_i(\vx)^\top(\vy-\vx) \leq \frac{\fH_\nu}{1 + \nu}\Norm{\vy-\vx}^{1+\nu},
    \end{align}
    for any $\vx, \vy\in\BR^d$ and  $i\in[n]$.
\end{lem}

\begin{proof}
    Following the proof of \cite{grapiglia2017regularized, grapiglia2019accelerated}, we have
    \begin{align*}
        f_i(\vy) - f_i(\vx) - \vg_i(\vx)^\top(\vy - \vx)
        & = \int_{t=0}^1 \vg_i(\vx + t(\vy - \vx))^\top(\vy - \vx)\text{d}t - \vg_i(\vx)^\top(\vy - \vx)\\
        & = \int_{t=0}^1 \left(\vg_i(\vx + t(\vy - \vx)) - \vg_i(\vx)\right)^\top(\vy - \vx)\text{d}t\\
        & \leq \int_{t=0}^1 \Norm{\vg_i(\vx + t(\vy - \vx)) - \vg_i(\vx)}\Norm{\vy - \vx}\text{d}t\\
        & \leq \int_{t=0}^1\fH_\nu t^\nu\Norm{\vy - \vx}^{1+\nu}\text{d}t\\
        & = \fH_\nu \Norm{\vy - \vx}^{1+\nu}\int_{t=0}^1 t^\nu\text{d}t\\
        & = \frac{\fH_\nu}{1 + \nu} \Norm{\vy - \vx}^{1+\nu},
    \end{align*}
    where the first inequality comes from Cauchy-Schwarz inequality, and the second one comes from Lemma \ref{le:H-con} that each gradient is H\"older continuous.
\end{proof}

\begin{lem}
\label{le:b-grad}
    \textnormal{(Bound for Jacobian and gradient)} Under Assumption \ref{asm:L-f}, we have
    \begin{align*}
        \Norm{\vg_i(\vx)}\leq\Norm{\mJ(\vx)}\leq L_f
    \end{align*}
    for all $ \vx\in\BR^d$ and $i\in [n]$.
\end{lem}

\begin{proof}

    For all $\vx, \vv\in\BR^d$, we have
    \begin{align*}
        \mJ(\vx)\vv =\lim_{h\to0}\frac{\vf(\vx+h\vv)-\vf(\vx)}{h}.
    \end{align*}
    Taking the spectral norm on both sides, we have
    \begin{align*}
        \Norm{\mJ(\vx)\vv}
        & = \lim_{h\to0}\frac{\Norm{\vf(\vx+h\vv)-\vf(\vx)}}{\vert h\vert}\\
        & \leq \lim_{h\to0}\frac{L_f\Norm{\vx+h\vv-\vx}}{\vert h\vert}\\
        & = \lim_{h\to0}\frac{L_f\vert h\vert\Norm{\vv}}{\vert h\vert}\\
        & = L_f\Norm{\vv},
    \end{align*}
    where the inequality comes from Assumption \ref{asm:L-f}. 
    
    Therefore, for all $\vx\in\BR^d$ it holds
    \begin{align*}
    \Norm{\mJ(\vx)}=\sup_{\vv\in\BR^d}\frac{\Norm{\mJ(\vx)\vv}}{\Norm{\vv}}\leq L_f.
    \end{align*}
    Let $\ve_i\in\BR^{n}$ be the $i$-th standard basic vector in $n$-dimensional Euclid space, then we have 

    \begin{align*}
    \Norm{\vg_i(\vx)} = \Norm{\mJ^\top \ve_i} \leq \Norm{\mJ}\Norm{\ve_i}=\Norm{\mJ(\vx)}\leq L_f
    \end{align*}
    for all $i\in[n]$.
\end{proof}

\section{The Auxiliary Sequence and Its Properties}
\label{sec:series}

We construct the following sequence for our convergence analysis in later sections.

\begin{dfn}
\label{def:series}
We define the following sequence $\{a_t(n,\nu)\}_{t\geq0}$ for given $n\in\BN^+$ and $\nu\in(0, 1]$:
\begin{align}
\label{eq:series}
a_t(n,\nu)\triangleq
\begin{cases}
    1,~~~~&t=0,\\[0.15cm]
    \displaystyle{\frac{1}{2(1+\nu)n}\left(\sum_{j=0}^{t-1}(a_j(n, \nu))^{1+\nu} + n - t\right)},~~~~& 1\leq t\leq n,\\[0.6cm]
     \displaystyle{\frac{1}{2(1+\nu)n}\sum_{j=t-n}^{t-1}(a_j(n, \nu))^{1+\nu}},~~~~&t> n.
\end{cases}
\end{align}
\end{dfn}

We then provide several useful properties for the sequence in Definition \ref{def:series}.
 
\begin{lem}
\label{le:s-mono-000} The sequence $\{a_t(n, \nu)\}_{t\geq 0}$ satisfies 
\begin{align*}
    a_t(n, \nu) \leq 1 
\end{align*}
for all $t\geq 0$. 
\end{lem}
\begin{proof}
\textbf{Part I:}
We first use induction to prove $a_t(n, \nu) \leq 1$ for all $t=0,1\dots,n$.
For the induction base, we can verify that $a_0(n, \nu) = 1 \leq 1$.  For the induction step, we assume 
\begin{align*}
a_j(n,\nu) \leq 1    
\end{align*}
holds for all $j=1,\dots,t-1$ such that $t\leq n$.
Then we have
\begin{align*}
        a_t(n,\nu) = \frac{1}{2(1+\nu)n}\left(\sum_{j=0}^{t-1}(a_j(n,\nu))^{1+\nu} + n - t\right) \leq \frac{1}{2(1+\nu)n}\left(t + n - t\right) = \frac{1}{2(1+\nu)}\leq 1,
\end{align*}
where the first inequality is based on the induction hypothesis and the last inequality is based on the setting $\nu\in(0,1]$. This finishes the induction. 

\textbf{Part II:} We then use induction to prove $a_t(n, \nu) \leq 1$ for all $t\geq n+1$. For the induction base, we can verify that
\begin{align*}
    a_{n+1}(n,\nu) = \frac{1}{2(1+\nu)n}\sum_{j=1}^{n}(a_j(n, \nu))^{1+\nu} \leq \frac{1}{2(1+\nu)n}\cdot n = \frac{1}{2(1+\nu)} \leq 1,
\end{align*}
where the first inequality is based on $a_t(n, \nu) \leq 1$ for all $t\leq n$ (which have shown in Part I), and the last inequality is based on the setting $\nu \in (0, 1]$. For the induction step, we assume
\begin{align*}
    a_{n+1}(n,\nu) \leq 1
\end{align*}
holds for all $j=n+2,\dots,t-1$ such that $t\geq n+3$. Then we have
\begin{align*}
    a_{t}(n,\nu) = \frac{1}{2(1+\nu)n}\sum_{j=t-n}^{t-1}(a_j(n, \nu))^{1+\nu} \leq \frac{1}{2(1+\nu)n}\cdot n = \frac{1}{2(1+\nu)} \leq 1,
\end{align*}

where the first inequality is based on the induction hypothesis and the last inequality is based on the setting $\nu \in (0, 1]$. This finishes the induction.

Combining the results of above two parts, we finish the proof of this lemma.
\end{proof}

\begin{lem}
\label{le:s-mono}The sequence $\{a_t(n, \nu)\}_{t\geq 0}$ satisfies 
\begin{align*}
    a_{t}(n, \nu) \geq a_{t+1}(n, \nu) 
\end{align*}
for all $t\geq 0$. 

\end{lem}
\begin{proof}
\textbf{Part I:} For $t=0$, the fact $\nu\in(0, 1]$ means
\begin{align*}
    a_1(n, \nu) = \frac{1}{2(1+\nu)} \leq 1 = a_0(n, \nu).
\end{align*}
\textbf{Part II:}
For all $t=1,\dots,n-1$, we have
\begin{align*}
        & a_{t+1}(n, \nu) - a_{t}(n, \nu) \\
        & = \frac{1}{2(1+\nu)n}\left(\sum_{j=0}^{t}(a_j(n, \nu))^{1+\nu} + n - t - 1\right) - \frac{1}{2(1+\nu)n}\left(\sum_{j=0}^{t-1}(a_j(n, \nu))^{1+\nu} + n - t\right)\\
        & = \frac{1}{2(1+\nu)n}\left((a_{t}(n, \nu))^{1+\nu} - 1\right)  \leq 0,
\end{align*}     
where the last inequality is based on Lemma \ref{le:s-mono-000}. This indicates $a_{t+1}\leq a_{t}$ for $t=1,\dots,n-1$. 

\textbf{Part III:} For all $t\geq n$, we use induction to prove $a_{t+1}(n, \nu)\leq a_{t}(n, \nu)$. 
For the induction base, we can verify that
    \begin{align*}
        & a_{n+1}(n, \nu) - a_n(n, \nu) \\
        & = \frac{1}{2(1+\nu)n}\sum_{j=1}^{n}(a_j(n, \nu))^{1+\nu} - \frac{1}{2(1+\nu)n}\left(\sum_{j=1}^{n-1}(a_j(n, \nu))^{1+\nu} + 1\right)\\
        & = \frac{1}{2(1+\nu)n}\left((a_n(n, \nu))^{1+\nu} - 1\right)\\
        & \leq 0,
    \end{align*}
    where the last inequality is based on Lemma \ref{le:s-mono-000}. 
    
    For the induction step, we assume
\begin{align*}
    a_{j+1}(n, \nu) \leq a_{j}(n, \nu)
\end{align*}
holds for all $j = n+1, \cdots, t-1$ such that $t\geq n+2$. Then we have
    \begin{align*}
        & a_{t+1}(n, \nu) - a_{t}(n, \nu) \\
         =&  \frac{1}{2(1+\nu)n}\sum_{j=t-n + 1}^{t}(a_j(n, \nu)) ^{1+\nu} - \frac{1}{2(1+\nu)n}\sum_{j=t-n}^{t-1}(a_j(n, \nu))^{1+\nu} \leq 0,
    \end{align*}
where the inequality is based on the induction hypothesis and the fact $a_{t+1}(n, \nu) \leq a_{t}(n, \nu)$ for all~$t \leq n - 1$ (which have shown in Part I).

Combining the results of above three parts, we finish the proof of this lemma.
\end{proof}

\begin{lem}
\label{le:n-qua-con}
     For the sequence $\{a_t(n, \nu)\}_{t\geq 0}$, we have
    \begin{align*}
    a_t(n, \nu) \leq \frac{1}{2(1+\nu)}(a_{t-n}(n, \nu))^{1+\nu}
    \end{align*}
    for all $t\geq n$.
\end{lem}
\begin{proof}
    For all $t\geq n$, the definition of $a_t(n, \nu)$ implies
    \begin{align*}
    & a_t(n, \nu) = \frac{1}{2(1+\nu)n}\left(\sum_{j=t-n}^{t-1}(a_j(n, \nu))^{1+\nu}\right)  \\
    & \leq \frac{1}{2(1+\nu)} \max\{(a_{t-n}(n, \nu))^{1+\nu},\cdots, (a_{t-1}(n, \nu))^{1+\nu}\}.
    \end{align*}
    Additionally, Lemma \ref{le:s-mono} implies for all $t\geq n$, we have
    \begin{align*}
    \max\{(a_{t-n}(n, \nu))^{1+\nu},\cdots, (a_{t-1}(n, \nu))^{1+\nu}\} = (a_{t-n}(n, \nu))^{1+\nu}, \quad t\geq n.
    \end{align*}
    Combining above results, we achieve
    \begin{align*}
    a_t(n, \nu) \leq \frac{1}{2(1+\nu)}(a_{t-n}(n, \nu))^{1+\nu}
    \end{align*}
    for all $t\geq n$.
\end{proof}

\begin{lem}
\label{le:linear-con} For the sequence $\{a_t(n, \nu)\}_{t\geq 0}$, we have 
\begin{align*}
a_{t+1}(n, \nu)\leq c_0 a_t(n, \nu)
\end{align*}
for all $t\geq n$, where
\begin{align*}
    c_0 = 1- \frac{1}{n}\left(1-\left(\frac{1}{2(1+\nu)}\right)^{1+\nu}\right).
\end{align*}
\end{lem}
\begin{proof}
    For all $t\geq n$, we have
    \begin{align}\label{eq:lemma9-ieq}
    a_t(n, \nu) \leq \frac{1}{2(1+\nu)}(a_{t-n}(n, \nu))^{1+\nu} \leq \frac{1}{2(1+\nu)}a_{t-n}(n, \nu), 
    \end{align}
    where the first inequality is based on Lemma \ref{le:n-qua-con} and the second one is based on Lemma \ref{le:s-mono-000}.
    
    Then we also have 
    \begin{align*}\small\begin{split}
        a_{t+1} 
        & = \frac{1}{2(1+\nu)n}\left(\sum_{j=t-n+1}^{t}(a_j(n, \nu))^{1+\nu}\right)\\
        & \leq \frac{1}{2(1+\nu)n}\left(\left(\frac{1}{2(1+\nu)}\right)^{1+\nu}(a_{t-n}(n, \nu))^{1+\nu} + \sum_{j=t-n+1}^{t-1}(a_j(n, \nu))^{1+\nu}\right)\\
        & = \frac{1}{2(1+\nu)n}\left(\left(\frac{1}{2(1+\nu)}\right)^{1+\nu}(a_{t-n}(n, \nu))^{1+\nu} + \sum_{j=t-n+1}^{t-1}(a_j(n, \nu))^{1+\nu} + (a_{t-n}(n, \nu))^{1+\nu} - (a_{t-n}(n, \nu))^{1+\nu} \right)\\
        & = a_{t}(n, \nu) + \frac{1}{2(1+\nu)n}\left(\left(\frac{1}{2(1+\nu)}\right)^{1+\nu}(a_{t-n}(n, \nu))^{1+\nu} - (a_{t-n}(n, \nu))^{1+\nu} \right)\\
        & = a_{t}(n, \nu) - \frac{1}{2(1+\nu)n}\left(1-\left(\frac{1}{2(1+\nu)}\right)^{1+\nu}\right)(a_{t-n}(n, \nu))^{1+\nu}\\
        & \leq a_{t}(n, \nu) - \frac{1}{n}\left(1-\left(\frac{1}{2(1+\nu)}\right)^{1+\nu}\right)a_t(n, \nu)\\
        & = \left(1- \frac{1}{n}\left(1-\left(\frac{1}{2(1+\nu)}\right)^{1+\nu}\right)\right)a_t(n, \nu),
     \end{split}
    \end{align*}
    for all $t\geq n$, where the first inequality is based on equation (\ref{eq:lemma9-ieq}) and the last inequality is based on Lemma \ref{le:n-qua-con}. This finish the proof.
\end{proof} 

\begin{lem}
\label{le:imp-c} For the sequence $\{a_t(n, \nu)\}_{t\geq 0}$, if there exists $c_1\in(0,1)$ and $t_0\geq 0$ such that
\begin{align}\label{eq:ieq-t0n2n}
a_{t+1}(n, \nu) \leq c_1 a_t(n, \nu)
\end{align}
for all $t\geq t_0 + n$, 
then we have
\begin{align*}
a_{t+1}(n, \nu) \leq c_1^{1+\nu} a_t(n, \nu)
\end{align*}
for all $t\geq t_0 + 2n$.
\end{lem}
\begin{proof}
    For all  $t\geq t_0 + 2n$, we have
    \begin{align*}
        & a_{t+1} (n, \nu)
         = \frac{1}{2(1+\nu)n}\sum_{j=t-n+1}^{t}(a_j(n, \nu))^{1+\nu} \\
        & \leq \frac{1}{2(1+\nu)n}\sum_{j=t-n}^{t-1}c_1^{1+\nu}(a_j(n, \nu))^{1+\nu} = c_1^{1+\nu} a_t(n, \nu),
    \end{align*}
    where the inequality is based on equation (\ref{eq:ieq-t0n2n}).
\end{proof}

\begin{lem}
\label{le:super-con} For the sequence $\{a_t(n, \nu)\}_{t\geq 0}$,
we have the superlinear convergence
\begin{align*}
a_{t+1}(n, \nu) \leq c^{(1+\nu)^{\left(\floor{\frac{t-1}{n}}-1\right)}}a_t(n, \nu) 
\end{align*}
for all $t\geq n$, where
\begin{align*}
    c = 1- \frac{1}{n}\left(1-\left(\frac{1}{2(1+\nu)}\right)^{1+\nu}\right).
\end{align*}
\begin{proof}
    According to Lemma \ref{le:linear-con}, we have
    \begin{align*}
        a_{t+1}(n, \nu)\leq c a_t(n, \nu) 
        \quad \text{for all}~~t \geq n.
    \end{align*}
    According to Lemma \ref{le:imp-c}, we have
    \begin{align*}
        a_{t+1}(n, \nu) & \leq c^{1+\nu} a_t(n, \nu) \quad \text{for all}~~t \geq 2n,\\
        a_{t+1}(n, \nu) & \leq c^{(1+\nu)^2} a_t(n, \nu) \quad \text{for all}~~t \geq 3n,\\
        a_{t+1}(n, \nu) & \leq c^{(1+\nu)^3} a_t(n, \nu) \quad \text{for all}~~t \geq 4n,\\
        \cdots &
    \end{align*}
    which implies
    \begin{align*}
        a_{t+1}(n, \nu) \leq c^{(1+\nu)^{\left(\lfloor {t}/{n}  \rfloor-1\right)}}a_t(n, \nu)
    \end{align*}
    for all $t \geq n$.

    The superlinear convergence of the sequence $\{a_t(n, \nu)\}_{t\geq 0}$ can be verify by the fact
    \begin{align*}
        \lim_{t\to\infty}c^{(1+\nu)^{\left(\lfloor {t}/{n}  \rfloor-1\right)}}=0.
    \end{align*}
    Hence, we finish the proof.
\end{proof}
\end{lem}

\section{The Convergence Analysis for \texttt{IGN}}
\label{sec:con-ign}

In this section, we provide the proofs for result in Section \ref{sec:ign}.

\subsection{The Proof of Proposition \ref{prop:b-JJ-sing}}

\begin{proof}
We denote the singular value decomposition of $\mJ(\vx^*)$ as
    \begin{align*}
        \mJ(\vx^*) = \mP\mD\mQ^\top,
    \end{align*}
    where $\mP\in\BR^{n\times d}, \mQ\in\BR^{d\times d}$ are (column) orthogonal matrices and $\mD\in\BR^{d\times d}$ is diagonal matrix with the smallest diagonal entry of $\mu>0$.
    Therefore, we have
    \begin{align*}
        \mJ(\vx^*)^\top \mJ(\vx^*) = \mQ\mD^2\mQ^\top,
    \end{align*}
    which means the smallest singular value of $\mJ(\vx^*)^\top \mJ(\vx^*)$ is equal to the smallest value of $\mD^2$, which is $\mu^2$. Therefore, we have
    \begin{align*}
        \sigma_{\min}(\mJ(\vx^*)^\top\mJ(\vx^*)) \geq \mu^2.
    \end{align*}
\end{proof}

\subsection{Proof of Lemma \ref{le:L-JJ}}

\begin{proof}
    The Jacobian holds that
    \begin{align*}
        \Norm{\mJ(\vy)^\top\mJ(\vy) - \mJ(\vx)^\top\mJ(\vx)}
        & = \Norm{\mJ(\vy)^\top\mJ(\vy) - \mJ(\vx)^\top\mJ(\vy) + \mJ(\vx)^\top\mJ(\vy) - \mJ(\vx)^\top\mJ(\vx)}\\
        & \leq \Norm{\left(\mJ(\vy)-\mJ(\vx)\right)^\top\mJ(\vy)} + \Norm{\mJ(\vx)^\top\left(\mJ(\vy) - \mJ(\vx)\right)}\\
        & \leq \Norm{\mJ(\vy)}\Norm{\mJ(\vy)-\mJ(\vx)} + \Norm{\mJ(\vx)}\Norm{\mJ(\vy)-\mJ(\vx)}\\
        & \leq 2L_f\fH_\nu\Norm{\vy-\vx}^\nu,
    \end{align*}
    where the first inequality comes from triangular inequality, the second inequality comes from property of norm, and the last inequality is based on Lemma \ref{le:b-grad} and Assumption \ref{asm:holder-g}. 
    
    For all $j\in[n]$, the gradient holds that
    \begin{align*}
        \Norm{\vg_i(\vy)\vg_i(\vy)^\top\!- \vg_i(\vx)\vg_i(\vx)^\top}
        & = \Norm{\vg_i(\vy)\vg_i(\vy)^\top\!- \vg_i(\vx)\vg_i(\vy)^\top + \vg_i(\vx)\vg_i(\vy)^\top\!- \vg_i(\vx)\vg_i(\vx)^\top}\\
        & \leq \Norm{\left(\vg_i(\vy)-\vg_i(\vx)\right)\vg_i(\vy)^\top} + \Norm{\vg_i(\vx)\left(\vg_i(\vy)-\vg_i(\vx)\right)^\top}\\
        & \leq \Norm{\vg_i(\vy)}\Norm{\vg_i(\vy)-\vg_i(\vx)} + \Norm{\vg_i(\vx)}\Norm{\vg_i(\vy)-\vg_i(\vx)}\\
        & \leq 2L_f\fH_\nu\Norm{\vy-\vx}^\nu,
    \end{align*}
    where the first inequality comes from triangular inequality, the second inequality comes from property of norm, and the last inequality is based on Lemma \ref{le:H-con} and \ref{le:b-grad}.
\end{proof}

\subsection{Proof of Lemma \ref{le:lb-Ht-new-1}}

\begin{proof}
    We have
    \begin{align*}
        \Norm{\mH^t - \mJ(\vx^*)^\top\mJ(\vx^*)} &= \Norm{\sum_{i=1}^n\vg_i(\vz_i^t)\vg_i(\vz_i^t)^\top - \sum_{i=1}^n\vg_i(\vx^*)\vg_i(\vx^*)^\top}\\
        & \leq \sum_{i=1}^n\Norm{\vg_i(\vz_i^t)\vg_i(\vz_i^t)^\top - \vg_i(\vx^*)\vg_i(\vx^*)^\top}\\
        & \leq \sum_{i=1}^n 2L_f\fH_\nu\Norm{\vz_i^t - \vx^*}^\nu,
    \end{align*}
    where the first inequality comes from the triangle inequality and the second inequality is based on Lemma \ref{le:L-JJ}. 
    Thus, we have
     \begin{align*}
         \mH^{t}-\mJ(\vx^*)^{\top}\mJ(\vx^*)\succeq - \sum_{i=1}^n 2L_f\fH_\nu\Norm{\vz_i^t - \vx^*}^\nu \cdot\mI, 
     \end{align*}
    which implies that
    \begin{align*}
        \sigma_{\min}(\mH^t) \geq \sigma_{\min}(\mJ(\vx^*)^\top\mJ(\vx^*)) - \sum_{i=1}^n 2L_f\fH_\nu\Norm{\vz_i^t - \vx^*}^\nu = \mu^2 - 2L_f\fH_\nu\sum_{i=1}^n \Norm{\vz_i^t - \vx^*}^\nu,
    \end{align*}
    where the last step is based on Proposition \ref{prop:b-JJ-sing}.
\end{proof}

\subsection{The Proof of Theorem \ref{thm:holder-IGN-1}}

We first show the update
\begin{align*}
\mG^{t+1} =\mG^{t} - \mG^{t}\mU^t(\mI + (\mV^t)^\top\mG^{t}\mU^t)^{-1}(\mV^t)^\top\mG^{t}    
\end{align*}
in \IGN~method (Line \ref{line:update-G} of Algorithm \ref{alg:IGN-1}) is well-defined if the matrices $\mH^t$ and $\mH^{t+1}$ are non-singular.

\begin{lem}
\label{le:via-update}
    Following the setting of Theorem \ref{thm:holder-IGN-1}, if the matrices $\mH^t$ and $\mH^{t+1}$ are non-singular, then the matrix $\mI + {\mV^t}^\top\mG^{t}\mU^t$ is also non-singular, where
    \begin{align*}
    \mU^t = \begin{bmatrix}
        - \vg_{i_t}(\vz_{i_t}^t) \!& 
        \vg_{i_t}(\vx^{t+1})
    \end{bmatrix}\in\BR^{d\times 2},~~
    \mV^t = \begin{bmatrix}
        \vg_{i_t}(\vz_{i_t}^t) \!& 
        \vg_{i_t}(\vx^{t+1})
    \end{bmatrix}\in\BR^{d\times 2}~~\text{and}~~i_t ={t\%n} + 1.
    \end{align*}
\end{lem}
\begin{proof}
    The recursion of $\mH^t$ and the definition of $\mU^t$ and $\mV^t$ imply
    \begin{align*}
        \mH^{t+1} = \mH^t - \vg_{i_t}(\vz_{i_t}^t)\vg_{i_t}(\vz_{i_t}^t)^\top + \vg_{i_t}(\vx^{t+1})\vg_{i_t}(\vx^{t+1})^\top = \mH^t + \mU^t{\mV^t}^\top.
    \end{align*}
    Since we assume matrices $\mH^t$ and $\mH^{t+1}$ are non-singular, applying the matrix determinant lemma \cite[Section 9.1.2]{petersen2008matrix} on above equation leads to
    \begin{align*}
        \det(\mH^{t+1})=\det(\mH^t + \mU^t{\mV^t}^\top) = \det(\mI + {\mV^t}^\top(\mH^t)^{-1}\mU^t)\det(\mH^t).
    \end{align*}
    Then the definition $\mG^t={\mH^t}^{-1}$ implies
    \begin{align*}
        \det(\mI + {\mV^t}^\top\mG^t\mU^t) = \det(\mI + {\mV^t}^\top{\mH^t}^{-1}\mU^t) \neq 0
    \end{align*}
    which finish the proofs.
\end{proof}

Then we show the non-singular assumption on $\{\mH^j\}_{j=0}^t$ can upper bound the distance $\Norm{\vx^{t+1} - \vx^*}$.

\begin{lem}\label{le:main-iter-1} 
Under Assumptions \ref{asm:L-f} and \ref{asm:holder-g}, we
assume matrices $\{\mH^j\}_{j=0}^t$ are non-singular and run \IGN~(Algorithm \ref{alg:IGN-1}), then it holds
\begin{align*}
    \Norm{\vx^{t+1} - \vx^*}
     \leq \frac{L_f\fH_\nu}{1 
     +\nu}\Norm{\mG^t}\sum_{i=1}^n\Norm{\vz_i^t-\vx^*}^{1+\nu},
    \end{align*}
    where $\mG^t=\left(\mH^t\right)^{-1}$.
\end{lem}
\begin{proof}
    Subtracting the term $\vx^*$ on both sides of equation (\ref{eq:ign-update-0}), we have
    \begin{align*}
    \vx^{t+1} - \vx^*
    & = \mG^t\left(\sum_{i=1}^n \vg_i(\vz_i^t)\vg_i(\vz_i^t)^\top(\vz_i^t-\vx^*) - \sum_{i=1}^n f_i(\vz_i^t)\vg_i(\vz_i^t)\right)\\
    & = \mG^t\left(\sum_{i=1}^n \vg_i(\vz_i^t)\vg_i(\vz_i^t)^\top(\vz_i^t-\vx^*) - \sum_{i=1}^n f_i(\vz_i^t)\vg_i(\vz_i^t) + \sum_{i=1}^n f_i(\vx^*)\vg_i(\vz_i^t)\right)\\
    & = \mG^t\sum_{i=1}^n\left( \vg_i(\vz_i^t)^\top(\vz_i^t-\vx^*) - f_i(\vz_i^t) + f_i(\vx^*)\right)\vg_i(\vz_i^t).
    \end{align*}

    Taking the norm on the both sides of above results, we have
    \begin{align*}
    \Norm{\vx^{t+1} - \vx^*}
    & = \Norm{\mG^t\sum_{i=1}^n\left( \vg_i(\vz_i^t)^\top(\vz_i^t-\vx^*) - f_i(\vz_i^t) + f_i(\vx^*)\right)\vg_i(\vz_i^t)}\\
    & \leq \Norm{\mG^t}\Norm{\sum_{i=1}^n\left( \vg_i(\vz_i^t)^\top(\vz_i^t-\vx^*) - f_i(\vz_i^t) + f_i(\vx^*)\right)\vg_i(\vz_i^t)}\\
    & \leq \frac{L_f\fH_\nu}{1 + \nu}\Norm{\mG^t}\sum_{i=1}^n\Norm{\vz_i^t-\vx^*}^{1+\nu}
    \end{align*}    
    where the first inequality comes from the property of matrix norm, the second inequality is based on Lemma \ref{le:power-bound} and \ref{le:b-grad}.
\end{proof}

We split the results of Theorem \ref{thm:holder-IGN-1} into two parts (i.e., Theorem \ref{thm:holder-IGN-1-1} and \ref{thm:holder-IGN-1-2}) and provide their proofs as follows.
Our analysis is based on the properties of our the auxiliary sequence constructed in Section \ref{sec:series}.

\begin{thm}
\label{thm:holder-IGN-1-1}
    Under the Assumption \ref{asm:L-f}, \ref{asm:holder-g} and \ref{asm:b-J-sing}, we run \IGN~(Algorithm \ref{alg:IGN-1}) with initialization $\vx^0\in\BR^d$ and~$\mH^0=\mJ(\vx^0)^\top\mJ(\vx^0)$ such that
    \begin{align*}
    \Norm{\vx^0 - \vx^*}\leq \left(\frac{\mu^2}{4L_f\fH_\nu n}\right)^{{1}/{\nu}},
    \end{align*} 
    then it holds 
    \begin{align*}
    \sigma_{\min}(\mI + (\mV^t)^\top (\mH^t)^{-1}\mU^t) > 0,\quad
    \mH^t \succeq \frac{\mu^2}{2}\mI\quad
    \text{and}\quad
    \Norm{\vx^{t}-\vx^*}\leq a_{t+1}(n, \nu)\Norm{\vx^0-\vx^*}
    \end{align*}
    for all $t \geq 0$, where sequence $\{a_t(n, \nu)\}_{t\geq 0}$ is defined in equation (\ref{eq:series}). 
\end{thm}
\begin{proof}

We first show
\begin{align}\label{eq:Hxtxstar0}
\mH^t \succeq \frac{\mu^2}{2}\mI\qquad
\text{and}\qquad
\Norm{\vx^{t}-\vx^*}\leq a_{t+1}(n, \nu)\Norm{\vx^0-\vx^*}
\end{align}
holds for all $t\geq 0$. We split the proof of results (\ref{eq:Hxtxstar0}) into the following three parts.

\textbf{Part I:} For $t=0$, 
the initialization and the fact $a_0=1$ leads to 
\begin{align*}
 \Norm{\vx^0 - \vx^*} = a_0(n, \nu)\Norm{\vx^0 - \vx^*}.
\end{align*}

\textbf{Part II:} For all $t=0,\cdots, n-1$, we use induction to prove the results of 
\begin{align}\label{eq:Hxtxstar}
\mH^t \succeq \frac{\mu^2}{2}\mI\qquad
\text{and}\qquad
\Norm{\vx^{t+1}-\vx^*}\leq a_{t+1}(n, \nu)\Norm{\vx^0-\vx^*}.
\end{align}
For the induction base, we can apply Lemma \ref{le:lb-Ht-new-1} to verify
\begin{align*}
    \sigma_{\min}(\mH^0) &\geq  \mu^2 - 2L_f\fH_\nu\sum_{i=1}^n \Norm{\vz_i^0 - \vx^*}^\nu\\
    & = \mu^2 - 2L_f\fH_\nu\sum_{i=1}^n \Norm{\vx^0 - \vx^*}^\nu\\
    & \geq \mu^2 - 2L_f\fH_\nu n \frac{\mu^2}{4L_f\fH_\nu n}\\
    & = \mu^2 - \frac{\mu^2}{2}\\
    & = \frac{\mu^2}{2}.
\end{align*}

This implies
\begin{align}\label{eq:boundG0}
\mH^0 \succeq \frac{\mu^2}{2} \qquad \text{and}\qquad
\Norm{\mG^0} = \Norm{(\mH^0)^{-1}} \leq \frac{2}{\mu^2}.
\end{align}
According to Lemma \ref{le:main-iter-1}, we have
\begin{align*}
    \Norm{\vx^1 - \vx^*} 
    &\leq \frac{L_f\fH_\nu}{1+\nu}\Norm{\mG^0}\sum_{i=1}^n\Norm{\vz_i^0-\vx^*}^{1+\nu}\\
    & \leq \frac{L_f\fH_\nu}{1+\nu}\cdot\frac{2}{\mu^2}\cdot\sum_{i=1}^n\Norm{\vz_i^0-\vx^*}^{1+\nu}\\
    & = \frac{L_f\fH_\nu}{1+\nu}\cdot\frac{2}{\mu^2}\cdot n\Norm{\vx^0-\vx^*}^{1+\nu}\\
    & \leq \frac{nL_f\fH_\nu}{1+\nu}\cdot\frac{2}{\mu^2}\cdot\frac{\mu^2}{4L_f\fH_\nu n}\Norm{\vx^0-\vx^*}\\
    & = \frac{1}{2(1+\nu)}\Norm{\vx^0-\vx^*}\\
    & = a_1(n, \nu)\Norm{\vx^0 - \vx^*},
\end{align*}
where the first inequality is based on equation (\ref{eq:boundG0}) and the second inequality is based on initial condition. 
Therefore, the induction base holds

For the induction step, we assume
\begin{align*}
    \mH^j \succeq \frac{\mu^2}{2} \mI \qquad\text{and}\qquad\Norm{\vx^{j+1} - \vx^*} \leq a_{j+1}(n, \nu)\Norm{\vx^0 - \vx^*}
\end{align*}
hold for all $j = 2, \cdots, t-1$ such that $t \leq n - 1$. Therefore, the update (\ref{eq:update-1}) means
\begin{align}\label{eq:vziii}
    \vz_i^{t}=
    \begin{cases}
        \vx^i,~~~~&1 \leq i\leq t,\\
        \vx^0,~~~~&t < i \leq n. 
    \end{cases}
\end{align}
The induction hypothesis leads to
\begin{align*}
    \Norm{\vx^j-\vx^*}^{\nu}\leq (a_j(n, \nu))^{\nu}\Norm{\vx^0-\vx^*}^{\nu} \leq  \Norm{\vx^0-\vx^*}^{\nu}\leq \frac{\mu^2}{4L_f\fH_\nu n}, 
\end{align*}
for $j = 1, \cdots, t-1$, where the second is based on Lemma \ref{le:s-mono-000} and the third comes from the initial condition. 
Combining with the result of (\ref{eq:vziii}), we achive
\begin{align*}
    \Norm{\vz_i^{t} - \vx^*}^{\nu} \leq \frac{\mu^2}{4L_f\fH_\nu n}.
\end{align*}

According to Lemma \ref{le:lb-Ht-new-1}, we have
\begin{align*}
    \sigma_{\min}(\mH^{t}) &\geq  \mu^2 - 2L_f\fH_\nu\sum_{i=1}^n \Norm{\vz_i^{t} - \vx^*}^\nu\\
    & \geq \mu^2 - 2L_f\fH_\nu n \frac{\mu^2}{4L_f\fH_\nu n}\\
    & = \mu^2 - \frac{\mu^2}{2}\\
    & = \frac{\mu^2}{2},
\end{align*}

where the second inequality comes from the initial condition.
Therefore, we have
\begin{align*}
    \mH^{t}  \succeq \frac{\mu^2}{2}\mI \qquad \text{and}\qquad
    \Norm{\mG^{t}}  = \Norm{(\mH^{t})^{-1}} \leq \frac{2}{\mu^2}.
\end{align*}
According to Lemma \ref{le:main-iter-1}, we have
\begin{align*}
    \Norm{\vx^{t+1} - \vx^*} 
    & \leq \frac{L_f\fH_\nu }{1+\nu}\Norm{\mG^{t}}\sum_{i=1}^n\Norm{\vz_i^{t}-\vx^*}^{1+\nu}\\
    & \leq \frac{L_f\fH_\nu }{1+\nu}\frac{2}{\mu^2}\sum_{i=1}^n\Norm{\vz_i^{t}-\vx^*}^{1+\nu}\\
    & \leq \frac{2L_f\fH_\nu}{(1+\nu)\mu^2}\left(\sum_{j=1}^{t}\Norm{\vx^{j}-\vx^*}^{1+\nu} + (n - t)\Norm{\vx^0 - \vx^*}^{1+\nu}\right)\\
    & \leq \frac{2L_f\fH_\nu }{(1+\nu)\mu^2}\left(\sum_{j=1}^{t} (a_j(n, \nu))^{1+\nu}\Norm{\vx^0 - \vx^*}^{1+\nu} + (n - t) \Norm{\vx^0 - \vx^*}^{1+\nu}\right)\\
    & \leq \frac{2L_f\fH_\nu }{(1+\nu)\mu^2}\frac{\mu^2}{4L_f\fH_\nu n}\left(\sum_{j=1}^{t}(a_j(n, \nu))^{1+\nu} + n - t\right)\Norm{\vx^0 - \vx^*}\\
    & = \frac{1}{2(1+\nu)n}\left(\sum_{j=1}^{t}(a_j(n, \nu))^{1+\nu} + n - t\right)\Norm{\vx^0 - \vx^*}\\
    & = \frac{1}{2(1+\nu)n}\left(\sum_{j=0}^{t}(a_j(n, \nu))^{1+\nu} + n - t - 1\right)\Norm{\vx^0 - \vx^*}\\
    & = a_{t+1}(n, \nu)\Norm{\vx^0 - \vx^*},
\end{align*}
where the last equality comes from the fact $a_0(n, \nu)=1$. 
Therefore, we finish the induction.

\textbf{Part III:} For all $t \geq n$, we use induction to prove 
\begin{align*}
\mH^t\succeq(\mu^2/2)\mI\qquad\text{and}\qquad\Norm{\vx^{t+1} - \vx^*}\leq a_{t+1}(n, \nu)\Norm{\vx^0 - \vx^*}.    
\end{align*}
For the induction base, we can verify that it holds (from the result of Part II) 
\begin{align*}
\mH^j \succeq \frac{\mu^2}{2}\mI \qquad \text{for all~~} j=0,\dots,n - 1,
\end{align*}
and
\begin{align*}
\Norm{\vx^j-\vx^*}\leq a_j(n, \nu)\Norm{\vx^0-\vx^*} \qquad \text{for all~~} j=1,\dots,n.
\end{align*}
Then we have
\begin{align*}
    \Norm{\vx^j-\vx^*}^{\nu}\leq (a_j(n, \nu))^{\nu}\Norm{\vx^0-\vx^*}^{\nu} \leq \Norm{\vx^0-\vx^*}^{\nu} \leq \frac{\mu^2}{4L_f\fH_\nu n}, \quad \text{for all~~} j=1,\dots,n,
\end{align*}
where the second inequality is based on Lemma \ref{le:s-mono-000} and the third inequality is based on the initial condition.

From Eq.~\ref{eq:update-1}, we have
\begin{align*}
    \vz_i^{n}=\vx^i \qquad\text{for all~~} i\in[n].
\end{align*}

Therefore, we have
\begin{align*}
    \Norm{\vz_i^{n} - \vx^*}^{\nu} \leq \frac{\mu^2}{4L_f\fH_\nu n}, \qquad\text{for all~~} i\in[n].
\end{align*}
According to Lemma \ref{le:lb-Ht-new-1}, we have
\begin{align*}
    \sigma_{\min}(\mH^{n}) &\geq  \mu^2 - 2L_f\fH_\nu\sum_{i=1}^n \Norm{\vz_i^{n} - \vx^*}^\nu\\
    & \geq \mu^2 - 2L_f\fH_\nu n \frac{\mu^2}{4L_f\fH_\nu n}\\
    & \geq \mu^2 - \frac{\mu^2}{2} = \frac{\mu^2}{2},
\end{align*}

which implies
\begin{align*}
    \mH^{n}  \succeq \frac{\mu^2}{2}\mI \qquad\text{and}\qquad
    \Norm{\mG^{n}}  = \Norm{(\mH^{n})^{-1}} \leq \frac{2}{\mu^2}.
\end{align*}
According to Lemma \ref{le:main-iter-1}, we have
\begin{align*}
    \Norm{\vx^{n+1} - \vx^*}
    & \leq \frac{L_f\fH_\nu }{1+\nu}\Norm{\mG^n}\sum_{i=1}^n\Norm{\vz_i^n-\vx^*}^{1+\nu}\\
    & \leq \frac{2L_f\fH_\nu }{(1+\nu)\mu^2}\sum_{i=1}^n\Norm{\vz_i^n-\vx^*}^{1+\nu}\\
    & \leq \frac{2L_f\fH_\nu }{(1+\nu)\mu^2}\left(\sum_{j=1}^n (a_j(n, \nu))^{1+\nu}\Norm{\vx^0 - \vx^*}^{1+\nu}\right)\\
    & \leq \frac{2L_f\fH_\nu }{(1+\nu)\mu^2}\frac{\mu^2}{4L_f\fH_\nu n}\left(\sum_{j=1}^{n}(a_j(n, \nu))^{1+\nu}\right)\Norm{\vx^0-\vx^*}\\
    & = \frac{1}{2(1+\nu)n}\left(\sum_{j=1}^{n}(a_j(n, \nu))^{1+\nu}\right)\Norm{\vx^0-\vx^*}\\
    & = a_{n+1}(n, \nu)\Norm{\vx^0-\vx^*}.
\end{align*}
Hence, we have shown the induction base holds. 

For the induction step, we assume
\begin{align*}
    \mH^j \succeq \frac{\mu^2}{2} \mI \qquad\text{and}\qquad\Norm{\vx^{j+1} - \vx^*} \leq a_{j+1}(n, \nu)\Norm{\vx^0 - \vx^*}
\end{align*}
holds for all $j = n+1, \cdots, t-1$ such that $t \geq n + 2$. Combining results of Part I and II, we have

\begin{align*}
\Norm{\vx^j-\vx^*}\leq a_j(n, \nu)\Norm{\vx^0-\vx^*} \qquad \text{for all}~~ j=0,\dots,t,
\end{align*}
which implies
\begin{align*}
    \Norm{\vx^j-\vx^*}^{\nu}\leq (a_j(n, \nu))^{\nu}\Norm{\vx^0-\vx^*}^{\nu} \leq \Norm{\vx^0-\vx^*}^{\nu} \leq \frac{\mu^2}{4L_f\fH_\nu n}, \qquad\text{for all}~~j=1,\dots,t,
\end{align*}
where the second inequality is based on Lemma \ref{le:s-mono-000} and the last inequality is based on the condition condition.

The update~(\ref{eq:update-1}) means the points $\{\vz_i^{t}\}_{i=1}^n$ can be written as $\{\vx^{t+1-n}, \cdots, \vx^{t}\}$, which implies
\begin{align*}
    \max\{\Norm{\vz_1^{t}-\vx^*}, \cdots, \Norm{\vz_n^{t}-\vx^*}\}=\max\{\Norm{\vx^{t+1-n}-\vx^*}, \cdots, \Norm{\vx^{t}-\vx^*}\}.
\end{align*}
Therefore, we have
\begin{align*}
    \Norm{\vz_i^{n} - \vx^*}^{\nu} \leq \frac{\mu^2}{4L_f\fH_\nu n} \qquad \text{for all}~~i=1,\dots,n.
\end{align*}

Combing with Lemma \ref{le:lb-Ht-new-1}, we have
\begin{align*}
    \sigma_{\min}(\mH^{t}) &\geq  \mu^2 - 2L_f\fH_\nu\sum_{i=1}^n \Norm{\vz_i^{t} - \vx^*}^\nu\\
    & \geq \mu^2 - 2L_f\fH_\nu n \frac{\mu^2}{4L_f\fH_\nu n}\\
    & = \mu^2 - \frac{\mu^2}{2} = \frac{\mu^2}{2}.
\end{align*}
Therefore, we achieve
\begin{align*}
    \mH^{t}  \succeq \frac{\mu^2}{2}\mI \qquad\text{and}\qquad
    \Norm{\mG^{t}} & = \Norm{(\mH^{t})^{-1}} \leq \frac{2}{\mu^2}.
\end{align*}
According to Lemma \ref{le:main-iter-1}, we have
\begin{align*}
    \Norm{\vx^{t+1} - \vx^*}
    & \leq \frac{L_f\fH_\nu }{1+\nu}\Norm{\mG^{t}}\sum_{i=1}^n\Norm{\vz_i^{t}-\vx^*}^{1+\nu}\\
    & \leq \frac{2L_f\fH_\nu }{(1+\nu)\mu^2}\sum_{i=1}^n\Norm{\vz_i^{t}-\vx^*}^{1+\nu}\\
    & \leq \frac{2L_f\fH_\nu }{(1+\nu)\mu^2}\left(\sum_{j=t-n+1}^{t} (a_j(n, \nu))^{1+\nu}\Norm{\vx^0 - \vx^*}^{1+\nu}\right)\\
    & \leq \frac{2L_f\fH_\nu }{(1+\nu)\mu^2}\left(\sum_{j=t-n+1}^{t} (a_j(n, \nu))^{1+\nu}\right)\Norm{\vx^0 - \vx^*}^{1+\nu}\\
    & \leq \frac{2L_f\fH_\nu }{(1+\nu)\mu^2}\frac{\mu^2}{4L_f\fH_\nu n}\left(\sum_{j=t-n+1}^{t+1}a_j(n, \nu)^{1+\nu}\right)\Norm{\vx^0 - \vx^*}\\
    & = \frac{1}{2(1+\nu)n}\left(\sum_{j=t-n+2}^{t+1}(a_j(n, \nu))^{1+\nu}\right)\Norm{\vx^0 - \vx^*}\\
    & = a_{t+1}(n, \nu)\Norm{\vx^0 - \vx^*}.
\end{align*}
Hence, we finish the induction.

Combining results of Part I, II and III completes the proof of (\ref{eq:Hxtxstar0}).

Since the non-singularity of $\mH^t$ and $\mH^{t+1}$ has been verified by result (\ref{eq:Hxtxstar0}), we can apply Lemma \ref{le:via-update} to achieve
\begin{align*}
    \sigma_{\min}(\mI + (\mV^t)^\top(\mH^t)^{-1}\mU^t) > 0.
\end{align*}
\end{proof}

\begin{thm}
\label{thm:holder-IGN-1-2}
We define the sequence $\{r_t\}_{t\geq 0}$ such that
\begin{align*}
    r_t\triangleq\begin{cases}
    \max\{\Norm{\vx^0 - \vx^*}, 1\},~~~~&t=0,\\[0.2cm]
    a_t(n, \nu)r_0,~~~~& t\geq 1,\\
\end{cases}
\end{align*}
where the sequence $\{a_t(n, \nu)\}_{t\geq 0}$ is defined by equation (\ref{eq:series}). 
Under the Assumptions \ref{asm:L-f}, \ref{asm:holder-g} and \ref{asm:b-J-sing}, running \IGN~(Algorithm \ref{alg:IGN-1}) with initial condition shown in Theorem \ref{thm:holder-IGN-1-1}, we have
    \begin{align}
    \label{eq:r_t_recursion}
    \Norm{\vx^t - \vx^*} \leq r_t
    \qquad\text{and}\qquad
    r_{t+1}\leq  c^{(1+\nu)^{\left(\floor{\frac{t}{n}}-1\right)}}r_t  
    \end{align}
    for all $t\geq n$, where
    \begin{align*}
    c = 1- \frac{1}{n}\left(1-\left(\frac{1}{2(1+\nu)}\right)^{1+\nu}\right).
    \end{align*}
\end{thm}

\begin{proof}
    The definition of $\{r_t\}_{t\geq 0}$ leads to 
    \begin{align*}
        r_0 = \max\{\Norm{\vx^0 - \vx^*}, 1\} \geq \Norm{\vx^0 - \vx^*}.
    \end{align*}
    According to Theorem \ref{thm:holder-IGN-1-1}, we have 
    \begin{align*}
    \Norm{\vx^t - \vx^*} \leq a_{t}(n, \nu)\Norm{\vx^0 - \vx^*} \leq a_t(n, \nu)r_0 = r_t.
    \end{align*}
    According to Lemma \ref{le:super-con}, we have
    \begin{align*}
         a_{t+1}(n, \nu)\leq  c^{(1+\nu)^{\left(\floor{\frac{t}{n}}-1\right)}}a_t(n, \nu) \qquad \text{for all}~~t \geq n.
    \end{align*}
    Thus, achieve
    \begin{align*}
        r_{t+1} = a_{t+1}(n, \nu)r_0\leq  c^{(1+\nu)^{\left(\floor{\frac{t}{n}}-1\right)}}a_t(n, \nu)r_0 =c^{(1+\nu)^{\left(\floor{\frac{t}{n}}-1\right)}}r_t  \qquad \text{for all}~~t \geq n,
    \end{align*}
    where
    \begin{align*}
        c = 1- \frac{1}{n}\left(1-\left(\frac{1}{2(1+\nu)}\right)^{1+\nu}\right).
    \end{align*}
\end{proof}

Combining the results of Theorem \ref{thm:holder-IGN-1-1} and  \ref{thm:holder-IGN-1-2}, we finish the proof of Theorem \ref{thm:holder-IGN-1}.

\subsection{Proof of Corollary \ref{cor:ign1}}

\begin{proof}
According to Theorem \ref{thm:holder-IGN-1}, we have
\begin{align*}
r_{t+1}\leq  c^{(1+\nu)^{\left(\floor{\frac{t}{n}}-1\right)}}r_t
\qquad\text{with}\qquad
c = 1- \frac{1}{n}\left(1-\left(\frac{1}{2(1+\nu)}\right)^{1+\nu}\right).
\end{align*}
for all $\nu\in(0,1]$.
Noticing that the value of $c$ is monotonically decreasing according to $\nu$, we have
\begin{align*}
    1 - \frac{1}{2n} > c \geq 1 - \frac{15}{16n},
\end{align*}
which implies
\begin{align*}
    r_{t+1} \leq \Big(1 - \frac{1}{2n}\Big)^{(1+\nu)^{(\floor{t/n}-1)}}r_t
\end{align*}
for all $t\geq n$.

\end{proof}

\subsection{Proof of Corollary \ref{cor:qua-1}}

\begin{proof}
According to the definition of $\{r_t\}_{t\geq 0}$ and Theorem \ref{thm:holder-IGN-1-2}, 
we have
\begin{align*}
    r_0 = \max\{\Norm{\vx^0 - \vx^*},1\} \geq 1.
\end{align*}
Combining with Lemma \ref{le:n-qua-con}, we have
\begin{align*}
r_t = & a_t(n, \nu) r_0 \\
\leq & \frac{1}{2(1+\nu)}(a_{t-n}(n, \nu))^{1+\nu}r_0 \\
=& \frac{1}{2(1+\nu)r_0^{\nu}}(a_{t-n}(n, \nu))^{1+\nu}r_0^{1+\nu} \\
=& \frac{1}{2(1+\nu)r_0^\nu}r_{t-n}^{1+\nu} \\
\leq & \frac{1}{2(1+\nu)}r_{t-n}^{1+\nu}
\end{align*}
for all $t \geq n$.
This leads to 
\begin{align*}
    r_t \leq \frac{1}{4} r_{t-n}^{2}
\end{align*}
in the case of $\nu=1$.
\end{proof}

\section{The Convergence Analysis for \texttt{MB-IGN}}
\label{sec:con-mbign}

In this section, we analyze the convergence of \texttt{MB-IGN} (Algorithm \ref{alg:MB-IGN}). 
Most of the proof in this section can be achieved by follow the analysis in Section \ref{sec:con-ign} and we provide the details for the completeness.

\subsection{The Additional Lemma for Gram Matrix}

We provide the bound for the spectrum of matrix $\mH^t$ for \MBIGN~method as follows

\begin{lem}
\label{le:lb-Ht-new-k}
Under Assumptions \ref{asm:L-f}, \ref{asm:holder-g} and \ref{asm:b-J-sing}, running \MBIGN~(Algorithm \ref{alg:MB-IGN}) with batch size $k$, $\mH^0 = \mJ(\vx^0)^\top\mJ(\vx^0)$ and~$\mG^0=(\mH^0)^{-1}$ holds that
\begin{align*}
    \sigma_{\min}(\mH^t) \geq \mu^2 - 2kL_f\fH_\nu\sum_{i=1}^m \Norm{\vz_i^t - \vx^*}^\nu
\end{align*}
for all $t\geq 0$, where $m=\ceil{n/k}$.
\end{lem}

\begin{proof}
    We have
    \begin{align*}
        \Norm{\mH^t - \mJ(\vx^*)^\top\mJ(\vx^*)} &= \Norm{\sum_{i=1}^m\sum_{j\in\fS_i}\vg_j(\vz_i^t)\vg_j(\vz_i^t)^\top - \sum_{i=1}^m\sum_{j\in\fS_i}\vg_j(\vx^*)\vg_j(\vx^*)^\top}\\
        & \leq \sum_{i=1}^m\sum_{j\in\fS_i}\Norm{\vg_j(\vz_i^t)\vg_j(\vz_i^t)^\top - \vg_j(\vx^*)\vg_j(\vx^*)^\top}\\
        & \leq \sum_{i=1}^m 2|\fS_i|L_f\fH_\nu\Norm{\vz_i^t - \vx^*}^\nu\\
        & \leq \sum_{i=1}^m 2kL_f\fH_\nu\Norm{\vz_i^t - \vx^*}^\nu,
    \end{align*}
    where the first inequality comes from the triangle inequality and the second inequality is based on Lemma \ref{le:L-JJ}. 
    Thus, we have
     \begin{align*}
         \mH^{t}-\mJ(\vx^*)^{\top}\mJ(\vx^*)\succeq - \sum_{i=1}^m 2kL_f\fH_\nu\Norm{\vz_i^t - \vx^*}^\nu \cdot\mI, 
     \end{align*}
    which implies that
    \begin{align*}
        \sigma_{\min}(\mH^t) \geq \sigma_{\min}(\mJ(\vx^*)^\top\mJ(\vx^*)) - \sum_{i=1}^m 2kL_f\fH_\nu\Norm{\vz_i^t - \vx^*}^\nu = \mu^2 -  2kL_f\fH_\nu\sum_{i=1}^m\Norm{\vz_i^t - \vx^*}^\nu,
    \end{align*}
    where the last step is based on Proposition \ref{prop:b-JJ-sing}.
\end{proof}

\subsection{Proof of Theorem \ref{thm:Group-1}}
\label{sec:proof-k}
Similarly, we then show the update
\begin{align*}
\mG^{t+1} =\mG^{t} - \mG^{t}\mU^t(\mI + (\mV^t)^\top\mG^{t}\mU^t)^{-1}(\mV^t)^\top\mG^{t}    
\end{align*}
in \MBIGN~method (Line \ref{line:update-G-k} of Algorithm \ref{alg:MB-IGN}) is well-defined if the matrices $\mH^t$ and $\mH^{t+1}$ are non-singular.

\begin{lem}
\label{le:via-update-k}
    Following the setting of Theorem \ref{thm:Group-1}, if the matrices $\mH^t$ and $\mH^{t+1}$ are non-singular, then the matrix $\mI + {\mV^t}^\top\mG^{t}\mU^t$ is also non-singular, where
    \begin{align*}
\begin{cases}    
\mU^t = \Big[- \vg_{j_1}(\vz_{i_t}^t),~~ 
    \vg_{j_1}(\vx^{t+1}),~\cdots~,~
    - \vg_{j_{|\fS_{i_t}|}}(\vz_{i_t}^t),~~
    \vg_{j_{|\fS_{i_t}|}}(\vx^{t+1})
\Big],\\[0.2cm]
\mV^t = \Big[\vg_{j_1}(\vz_{i_t}^t),~~ 
    \vg_{j_1}(\vx^{t+1}),~\cdots~,~
    \vg_{j_{|\fS_{i_t}|}}(\vz_{i_t}^t),~~
    \vg_{j_{|\fS_{i_t}|}}(\vx^{t+1})
\Big], \quad i_t ={t\%m} + 1,
\end{cases}
\end{align*}
\end{lem}
\begin{proof}
    The recursion of $\mH^t$ and the definition of $\mU^t$ and $\mV^t$ imply
    \begin{align*}
        \mH^{t+1} = \mH^t - \sum_{j\in\fS_{i_t}}\vg_{j}(\vz_{i_t}^t)\vg_{j}(\vz_{i_t}^t)^\top + \sum_{j\in\fS_{i_t}}\vg_{j}(\vx^{t+1})\vg_{j}(\vx^{t+1})^\top = \mH^t + \mU^t{\mV^t}^\top.
    \end{align*}
    Since we assume matrices $\mH^t$ and $\mH^{t+1}$ are non-singular, applying the matrix determinant lemma \cite[section 9.1.2]{petersen2008matrix} on above equation leads to
    \begin{align*}
        \det(\mH^{t+1})=\det(\mH^t + \mU^t{\mV^t}^\top) = \det(\mI + {\mV^t}^\top(\mH^t)^{-1}\mU^t)\det(\mH^t).
    \end{align*}
    Then the definition $\mG^t={\mH^t}^{-1}$ implies
    \begin{align*}
        \det(\mI + {\mV^t}^\top\mG^t\mU^t) = \det(\mI + {\mV^t}^\top{\mH^t}^{-1}\mU^t) \neq 0
    \end{align*}
    which finish the proofs.
\end{proof}

Then we show the non-singular assumption on $\{\mH^j\}_{j=0}^t$ can upper bound the distance $\Norm{\vx^{t+1} - \vx^*}$.

\begin{lem}\label{le:main-iter-k} 
Under Assumptions \ref{asm:L-f} and \ref{asm:holder-g}, we
assume matrices $\{\mH^j\}_{j=0}^t$ are non-singular and run \MBIGN~(Algorithm \ref{alg:MB-IGN}) with batch size $k$, then it holds
\begin{align*}
    \Norm{\vx^{t+1} - \vx^*}
     \leq \frac{kL_f\fH_\nu}{1 
     +\nu}\Norm{\mG^t}\sum_{i=1}^m\Norm{\vz_i^t-\vx^*}^{1+\nu},
    \end{align*}
    where $\mG^t=\left(\mH^t\right)^{-1}$ and $m=\ceil{n/k}$.
\end{lem}
\begin{proof}
    Subtracting the term $\vx^*$ on both sides of equation (\ref{eq:ign-update-0}), we have
    \begin{align*}
    \vx^{t+1} - \vx^*
        & = \left(\sum_{i=1}^{m}\sum_{j\in\fS_i}\vg_j(\vz_i^t)\vg_j(\vz_i^t)^\top \right)^{-1}\left(\sum_{i=1}^{m}\left(\sum_{j\in\fS_i}\vg_j(\vz_i^t)\vg_j(\vz_i^t)^\top\right) (\vz_i^t - \vx^*) - \sum_{i=1}^{m}\sum_{j\in\fS_i} f_j(\vz_i^t)\vg_j(\vz_i^t)\right)\\
    & = \mG^t\left(\sum_{i=1}^{m}\left(\sum_{j\in\fS_i}\vg_j(\vz_i^t)\vg_j(\vz_i^t)^\top\right) (\vz_i^t - \vx^*) - \sum_{i=1}^{m}\sum_{j\in\fS_i} f_j(\vz_i^t)\vg_j(\vz_i^t) + \sum_{i=1}^{m}\sum_{j\in\fS_i} f_j(\vx^*)\vg_j(\vz_i^t)\right)\\
    & = \mG^t\sum_{i=1}^{m}\sum_{j\in\fS_i}\vg_j(\vz_i^t)\left( \vg_j(\vz_i^t)^\top(\vz_i^t-\vx^*) - f_j(\vz_i^t) + f_j(\vx^*)\right).
    \end{align*}

    Taking the norm on the both sides of above results, we have
    \begin{align*}
    \Norm{\vx^{t+1} - \vx^*}
    & = \Norm{\mG^t\sum_{i=1}^{m}\sum_{j\in\fS_i}\vg_j(\vz_i^t)\left( \vg_j(\vz_i^t)^\top(\vz_i^t-\vx^*) - f_j(\vz_i^t) + f_j(\vx^*)\right)}\\
    & \leq \Norm{\mG^t}\Norm{\sum_{i=1}^m\sum_{j\in\fS_i}\left( \vg_i(\vz_i^t)^\top(\vz_i^t-\vx^*) - f_i(\vz_i^t) + f_i(\vx^*)\right)\vg_i(\vz_i^t)}\\
    & \leq \frac{L_f\fH_\nu}{1 + \nu}\Norm{\mG^t}\sum_{i=1}^n\sum_{j\in\fS_i}\Norm{\vz_i^t-\vx^*}^{1+\nu}\\
    & \leq \frac{kL_f\fH_\nu}{1 + \nu}\Norm{\mG^t}\sum_{i=1}^n\Norm{\vz_i^t-\vx^*}^{1+\nu}
    \end{align*}    
    where the first inequality comes from the property of matrix norm, the second inequality is based on Lemma \ref{le:power-bound} and \ref{le:b-grad}, the last inequality is based on $|\fS_i|\leq k$ for all $i\in[m]$.
\end{proof}

We split the results of Theorem \ref{thm:Group-1} into two parts (i.e., Theorem \ref{thm:Group-1-1} and \ref{thm:Group-1-2}) and provide their proofs as follows.
Our analysis is based on the properties of our the auxiliary sequence constructed in Section~\ref{sec:series}.

\begin{thm}
\label{thm:Group-1-1}
    Under the Assumption \ref{asm:L-f}, \ref{asm:holder-g} and \ref{asm:b-J-sing}, we run \MBIGN~(Algorithm \ref{alg:MB-IGN}) with batch size $k$, and initialization $\vx^0\in\BR^d$ and~$\mH^0=\mJ(\vx^0)^\top\mJ(\vx^0)$ such that
    \begin{align*}
    \Norm{\vx^0 - \vx^*}\leq \left(\frac{\mu^2}{4kL_f\fH_\nu m}\right)^{{1}/{\nu}},
    \end{align*} 
    where $m=\ceil{n/k}$, then it holds 
    \begin{align*}
    \sigma_{\min}(\mI + (\mV^t)^\top (\mH^t)^{-1}\mU^t) > 0,\quad
    \mH^t \succeq \frac{\mu^2}{2}\mI\quad
    \text{and}\quad
    \Norm{\vx^{t}-\vx^*}\leq a_{t+1}(m, \nu)\Norm{\vx^0-\vx^*}
    \end{align*}
    for all $t \geq 0$, where the sequence $\{a_t(m, \nu)\}_{t\geq 0}$ is defined in equation (\ref{eq:series}). 
\end{thm}

\begin{proof}
We first show
\begin{align}\label{eq:Hxtxstar0-k}
\mH^t \succeq \frac{\mu^2}{2}\mI\qquad
\text{and}\qquad
\Norm{\vx^{t}-\vx^*}\leq a_{t+1}(m, \nu)\Norm{\vx^0-\vx^*}
\end{align}
holds for all $t\geq 0$. We split the proof of results (\ref{eq:Hxtxstar0-k}) into the following three parts.

\textbf{Part I:} For $t=0$, 
the initialization and the fact $a_0=1$ leads to 
\begin{align*}
 \Norm{\vx^0 - \vx^*} = a_0(m, \nu)\Norm{\vx^0 - \vx^*}.
\end{align*}

\textbf{Part II:} For all $t=0,\cdots, m-1$, we use induction to prove the results of 
\begin{align}\label{eq:Hxtxstar-k}
\mH^t \succeq \frac{\mu^2}{2}\mI\qquad
\text{and}\qquad
\Norm{\vx^{t+1}-\vx^*}\leq a_{t+1}(m, \nu)\Norm{\vx^0-\vx^*}.
\end{align}
For the induction base, we can apply Lemma \ref{le:lb-Ht-new-k} to verify
\begin{align*}
    \sigma_{\min}(\mH^0) &\geq  \mu^2 - 2kL_f\fH_\nu\sum_{i=1}^m \Norm{\vz_i^0 - \vx^*}^\nu\\
    & = \mu^2 - 2kL_f\fH_\nu\sum_{i=1}^m \Norm{\vx^0 - \vx^*}^\nu\\
    & \geq \mu^2 - 2kL_f\fH_\nu m \frac{\mu^2}{4kL_f\fH_\nu m}\\
    & = \mu^2 - \frac{\mu^2}{2}\\
    & = \frac{\mu^2}{2}.
\end{align*}

This implies
\begin{align}\label{eq:boundG0-k}
\mH^0 \succeq \frac{\mu^2}{2} \qquad \text{and}\qquad
\Norm{\mG^0} = \Norm{(\mH^0)^{-1}} \leq \frac{2}{\mu^2}.
\end{align}
According to Lemma \ref{le:main-iter-k}, we have
\begin{align*}
    \Norm{\vx^1 - \vx^*} 
    &\leq \frac{kL_f\fH_\nu}{1+\nu}\Norm{\mG^0}\sum_{i=1}^m\Norm{\vz_i^0-\vx^*}^{1+\nu}\\
    & \leq \frac{kL_f\fH_\nu}{1+\nu}\cdot\frac{2}{\mu^2}\cdot\sum_{i=1}^m\Norm{\vz_i^0-\vx^*}^{1+\nu}\\
    & = \frac{kL_f\fH_\nu}{1+\nu}\cdot\frac{2}{\mu^2}\cdot m\Norm{\vx^0-\vx^*}^{1+\nu}\\
    & \leq \frac{kmL_f\fH_\nu}{1+\nu}\cdot\frac{2}{\mu^2}\cdot\frac{\mu^2}{4kL_f\fH_\nu m}\Norm{\vx^0-\vx^*}\\
    & = \frac{1}{2(1+\nu)}\Norm{\vx^0-\vx^*}\\
    & = a_1(m, \nu)\Norm{\vx^0 - \vx^*},
\end{align*}
where the first inequality is based on equation (\ref{eq:boundG0-k}) and the second inequality is based on initial condition. 
Therefore, the induction base holds

For the induction step, we assume
\begin{align*}
    \mH^j \succeq \frac{\mu^2}{2} \mI \qquad\text{and}\qquad\Norm{\vx^{j+1} - \vx^*} \leq a_{j+1}(m, \nu)\Norm{\vx^0 - \vx^*}
\end{align*}
hold for all $j = 2, \cdots, t-1$ such that $t \leq m - 1$. Therefore, the update (\ref{eq:update-1}) means
\begin{align}\label{eq:vziii-k}
    \vz_i^{t}=
    \begin{cases}
        \vx^i,~~~~&1 \leq i\leq t,\\
        \vx^0,~~~~&t < i \leq m. 
    \end{cases}
\end{align}
The induction hypothesis leads to
\begin{align*}
    \Norm{\vx^j-\vx^*}^{\nu}\leq (a_j(m, \nu))^{\nu}\Norm{\vx^0-\vx^*}^{\nu} \leq  \Norm{\vx^0-\vx^*}^{\nu}\leq \frac{\mu^2}{4kL_f\fH_\nu m}, 
\end{align*}
for $j = 1, \cdots, t-1$, where the second is based on Lemma \ref{le:s-mono-000} and the third comes from the initial condition. 
Combining with the result of (\ref{eq:vziii-k}), we achive
\begin{align*}
    \Norm{\vz_i^{t} - \vx^*}^{\nu} \leq \frac{\mu^2}{4kL_f\fH_\nu m}.
\end{align*}

According to Lemma \ref{le:lb-Ht-new-k}, we have
\begin{align*}
    \sigma_{\min}(\mH^{t}) &\geq  \mu^2 - 2kL_f\fH_\nu\sum_{i=1}^m \Norm{\vz_i^{t} - \vx^*}^\nu\\
    & \geq \mu^2 - 2kL_f\fH_\nu m \frac{\mu^2}{4kL_f\fH_\nu m}\\
    & = \mu^2 - \frac{\mu^2}{2}\\
    & = \frac{\mu^2}{2},
\end{align*}

where the second inequality comes from the initial condition.
Therefore, we have
\begin{align*}
    \mH^{t}  \succeq \frac{\mu^2}{2}\mI \qquad \text{and}\qquad
    \Norm{\mG^{t}}  = \Norm{(\mH^{t})^{-1}} \leq \frac{2}{\mu^2}.
\end{align*}
According to Lemma \ref{le:main-iter-1}, we have
\begin{align*}
    \Norm{\vx^{t+1} - \vx^*} 
    & \leq \frac{kL_f\fH_\nu }{1+\nu}\Norm{\mG^{t}}\sum_{i=1}^m\Norm{\vz_i^{t}-\vx^*}^{1+\nu}\\
    & \leq \frac{kL_f\fH_\nu }{1+\nu}\frac{2}{\mu^2}\sum_{i=1}^m\Norm{\vz_i^{t}-\vx^*}^{1+\nu}\\
    & \leq \frac{2L_f\fH_\nu}{(1+\nu)\mu^2}\left(\sum_{j=1}^{t}\Norm{\vx^{j}-\vx^*}^{1+\nu} + (m - t)\Norm{\vx^0 - \vx^*}^{1+\nu}\right)\\
    & \leq \frac{2kL_f\fH_\nu }{(1+\nu)\mu^2}\left(\sum_{j=1}^{t} (a_j(m, \nu))^{1+\nu}\Norm{\vx^0 - \vx^*}^{1+\nu} + (m - t) \Norm{\vx^0 - \vx^*}^{1+\nu}\right)\\
    & \leq \frac{2kL_f\fH_\nu }{(1+\nu)\mu^2}\frac{\mu^2}{4kL_f\fH_\nu m}\left(\sum_{j=1}^{t}(a_j(m, \nu))^{1+\nu} + m - t\right)\Norm{\vx^0 - \vx^*}\\
    & = \frac{1}{2(1+\nu)m}\left(\sum_{j=1}^{t}(a_j(m, \nu))^{1+\nu} + m - t\right)\Norm{\vx^0 - \vx^*}\\
    & = \frac{1}{2(1+\nu)m}\left(\sum_{j=0}^{t}(a_j(m, \nu))^{1+\nu} + m - t - 1\right)\Norm{\vx^0 - \vx^*}\\
    & = a_{t+1}(m, \nu)\Norm{\vx^0 - \vx^*},
\end{align*}
where the last equality comes from the fact $a_0(m, \nu)=1$. 
Therefore, we finish the induction.

\textbf{Part III:} For all $t \geq m$, we use induction to prove 
\begin{align*}
\mH^t\succeq(\mu^2/2)\mI\qquad\text{and}\qquad\Norm{\vx^{t+1} - \vx^*}\leq a_{t+1}(m, \nu)\Norm{\vx^0 - \vx^*}.    
\end{align*}
For the induction base, we can verify that it holds (from the result of Part II) 
\begin{align*}
\mH^j \succeq \frac{\mu^2}{2}\mI \qquad \text{for all~~} j=0,\dots,m - 1,
\end{align*}
and
\begin{align*}
\Norm{\vx^j-\vx^*}\leq a_j(m, \nu)\Norm{\vx^0-\vx^*} \qquad \text{for all~~} j=1,\dots,m.
\end{align*}
Then we have
\begin{align*}
    \Norm{\vx^j-\vx^*}^{\nu}\leq (a_j(m, \nu))^{\nu}\Norm{\vx^0-\vx^*}^{\nu} \leq \Norm{\vx^0-\vx^*}^{\nu} \leq \frac{\mu^2}{4kL_f\fH_\nu m}, \quad \text{for all~~} j=1,\dots,m,
\end{align*}
where the second inequality is based on Lemma \ref{le:s-mono-000} and the third inequality is based on the initial condition.

From Eq.~\ref{eq:update-1}, we have
\begin{align*}
    \vz_i^{m}=\vx^i \qquad\text{for all~~} i\in[m].
\end{align*}

Therefore, we have
\begin{align*}
    \Norm{\vz_i^{m} - \vx^*}^{\nu} \leq \frac{\mu^2}{4kL_f\fH_\nu m}, \qquad\text{for all~~} i\in[m].
\end{align*}
According to Lemma \ref{le:lb-Ht-new-k}, we have
\begin{align*}
    \sigma_{\min}(\mH^{m}) &\geq  \mu^2 - 2kL_f\fH_\nu\sum_{i=1}^m \Norm{\vz_i^{m} - \vx^*}^\nu\\
    & \geq \mu^2 - 2kL_f\fH_\nu m \frac{\mu^2}{4kL_f\fH_\nu m}\\
    & \geq \mu^2 - \frac{\mu^2}{2} = \frac{\mu^2}{2},
\end{align*}

which implies
\begin{align*}
    \mH^{m}  \succeq \frac{\mu^2}{2}\mI \qquad\text{and}\qquad
    \Norm{\mG^{m}}  = \Norm{(\mH^{m})^{-1}} \leq \frac{2}{\mu^2}.
\end{align*}
According to Lemma \ref{le:main-iter-k}, we have
\begin{align*}
    \Norm{\vx^{n+1} - \vx^*}
    & \leq \frac{kL_f\fH_\nu }{1+\nu}\Norm{\mG^n}\sum_{i=1}^m\Norm{\vz_i^m-\vx^*}^{1+\nu}\\
    & \leq \frac{2kL_f\fH_\nu }{(1+\nu)\mu^2}\sum_{i=1}^m\Norm{\vz_i^m-\vx^*}^{1+\nu}\\
    & \leq \frac{2kL_f\fH_\nu }{(1+\nu)\mu^2}\left(\sum_{j=1}^m (a_j(m, \nu))^{1+\nu}\Norm{\vx^0 - \vx^*}^{1+\nu}\right)\\
    & \leq \frac{2kL_f\fH_\nu }{(1+\nu)\mu^2}\frac{\mu^2}{4kL_f\fH_\nu m}\left(\sum_{j=1}^{m}(a_j(m, \nu))^{1+\nu}\right)\Norm{\vx^0-\vx^*}\\
    & = \frac{1}{2(1+\nu)m}\left(\sum_{j=1}^{m}(a_j(m, \nu))^{1+\nu}\right)\Norm{\vx^0-\vx^*}\\
    & = a_{m+1}(m, \nu)\Norm{\vx^0-\vx^*}.
\end{align*}
Hence, we have shown the induction base holds. 

For the induction step, we assume
\begin{align*}
    \mH^j \succeq \frac{\mu^2}{2} \mI \qquad\text{and}\qquad\Norm{\vx^{j+1} - \vx^*} \leq a_{j+1}(m, \nu)\Norm{\vx^0 - \vx^*}
\end{align*}
holds for all $j = m+1, \cdots, t-1$ such that $t \geq m + 2$. Combining results of Part I and II, we have

\begin{align*}
\Norm{\vx^j-\vx^*}\leq a_j(m, \nu)\Norm{\vx^0-\vx^*} \qquad \text{for all}~~ j=0,\dots,t,
\end{align*}
which implies
\begin{align*}
    \Norm{\vx^j-\vx^*}^{\nu}\leq (a_j(m, \nu))^{\nu}\Norm{\vx^0-\vx^*}^{\nu} \leq \Norm{\vx^0-\vx^*}^{\nu} \leq \frac{\mu^2}{4kL_f\fH_\nu m}, \qquad\text{for all}~~j=1,\dots,t,
\end{align*}
where the second inequality is based on Lemma \ref{le:s-mono-000} and the last inequality is based on the condition condition.

The update~(\ref{eq:update-2}) means the points $\{\vz_i^{t}\}_{i=1}^m$ can be written as $\{\vx^{t+1-m}, \cdots, \vx^{t}\}$, which implies
\begin{align*}
    \max\{\Norm{\vz_1^{t}-\vx^*}, \cdots, \Norm{\vz_m^{t}-\vx^*}\}=\max\{\Norm{\vx^{t+1-m}-\vx^*}, \cdots, \Norm{\vx^{t}-\vx^*}\}.
\end{align*}
Therefore, we have
\begin{align*}
    \Norm{\vz_i^{m} - \vx^*}^{\nu} \leq \frac{\mu^2}{4kL_f\fH_\nu m} \qquad \text{for all}~~i=1,\dots,m.
\end{align*}

Combing with Lemma \ref{le:lb-Ht-new-k}, we have
\begin{align*}
    \sigma_{\min}(\mH^{t}) &\geq  \mu^2 - 2kL_f\fH_\nu\sum_{i=1}^n \Norm{\vz_i^{t} - \vx^*}^\nu\\
    & \geq \mu^2 - 2kL_f\fH_\nu m \frac{\mu^2}{4kL_f\fH_\nu m}\\
    & = \mu^2 - \frac{\mu^2}{2} = \frac{\mu^2}{2}.
\end{align*}
Therefore, we achieve
\begin{align*}
    \mH^{t}  \succeq \frac{\mu^2}{2}\mI \qquad\text{and}\qquad
    \Norm{\mG^{t}} & = \Norm{(\mH^{t})^{-1}} \leq \frac{2}{\mu^2}.
\end{align*}
According to Lemma \ref{le:main-iter-k}, we have
\begin{align*}
    \Norm{\vx^{t+1} - \vx^*}
    & \leq \frac{kL_f\fH_\nu }{1+\nu}\Norm{\mG^{t}}\sum_{i=1}^m\Norm{\vz_i^{t}-\vx^*}^{1+\nu}\\
    & \leq \frac{2kL_f\fH_\nu }{(1+\nu)\mu^2}\sum_{i=1}^m\Norm{\vz_i^{t}-\vx^*}^{1+\nu}\\
    & \leq \frac{2kL_f\fH_\nu }{(1+\nu)\mu^2}\left(\sum_{j=t-m+1}^{t} (a_j(m, \nu))^{1+\nu}\Norm{\vx^0 - \vx^*}^{1+\nu}\right)\\
    & \leq \frac{2kL_f\fH_\nu }{(1+\nu)\mu^2}\left(\sum_{j=t-m+1}^{t} (a_j(m, \nu))^{1+\nu}\right)\Norm{\vx^0 - \vx^*}^{1+\nu}\\
    & \leq \frac{2kL_f\fH_\nu }{(1+\nu)\mu^2}\frac{\mu^2}{4kL_f\fH_\nu m}\left(\sum_{j=t-m+1}^{t+1}(a_j(m, \nu))^{1+\nu}\right)\Norm{\vx^0 - \vx^*}\\
    & = \frac{1}{2(1+\nu)m}\left(\sum_{j=t-m+2}^{t+1}(a_j(m, \nu))^{1+\nu}\right)\Norm{\vx^0 - \vx^*}\\
    & = a_{t+1}(m, \nu)\Norm{\vx^0 - \vx^*}.
\end{align*}
Hence, we finish the induction.

Combining results of Part I, II and III completes the proof of (\ref{eq:Hxtxstar0-k}).

Since the non-singularity of $\mH^t$ and $\mH^{t+1}$ has been verified by result (\ref{eq:Hxtxstar0-k}), we can apply Lemma \ref{le:via-update-k} to achieve
\begin{align*}
    \sigma_{\min}(\mI + (\mV^t)^\top(\mH^t)^{-1}\mU^t) > 0.
\end{align*}
\end{proof}

\begin{thm}
\label{thm:Group-1-2}
We define the sequence $\{r_t\}_{t\geq 0}$ such that
\begin{align*}
    r_t\triangleq\begin{cases}
    \max\{\Norm{\vx^0 - \vx^*}, 1\},~~~~&t=0,\\[0.2cm]
    a_t(m, \nu)r_0,~~~~& t\geq 1,\\
\end{cases}
\end{align*}
where the sequence $\{a_t(m, \nu)\}_{t\geq 0}$ is defined by equation (\ref{eq:series}). 
Under the Assumptions \ref{asm:L-f}, \ref{asm:holder-g} and \ref{asm:b-J-sing}, running \MBIGN~(Algorithm \ref{alg:MB-IGN}) with initial condition shown in Theorem \ref{thm:Group-1-2}, we have
    \begin{align}
    \label{eq:r_t_recursion-k}
    \Norm{\vx^t - \vx^*} \leq r_t
    \qquad\text{and}\qquad
    r_{t+1}\leq  c^{(1+\nu)^{\left(\floor{\frac{t}{m}}-1\right)}}r_t  
    \end{align}
    for all $t\geq m$, where
    \begin{align*}
    c = 1- \frac{1}{m}\left(1-\left(\frac{1}{2(1+\nu)}\right)^{1+\nu}\right).
    \end{align*}
\end{thm}

\begin{proof}
    The definition of $\{r_t\}_{t\geq 0}$ leads to 
    \begin{align*}
        r_0 = \max\{\Norm{\vx^0 - \vx^*}, 1\} \geq \Norm{\vx^0 - \vx^*}.
    \end{align*}
    According to Theorem \ref{thm:Group-1-1}, we have 
    \begin{align*}
    \Norm{\vx^t - \vx^*} \leq a_{t}(m, \nu)\Norm{\vx^0 - \vx^*} \leq a_t(m, \nu)r_0 = r_t.
    \end{align*}
    According to Lemma \ref{le:super-con}, we have
    \begin{align*}
         a_{t+1}(m, \nu)\leq  c^{(1+\nu)^{\left(\floor{\frac{t}{m}}-1\right)}}a_t(m, \nu) \qquad \text{for all}~~t \geq m.
    \end{align*}
    Thus, achieve
    \begin{align*}
        r_{t+1} = a_{t+1}(m, \nu)r_0\leq  c^{(1+\nu)^{\left(\floor{\frac{t}{m}}-1\right)}}a_t(m, \nu)r_0 =c^{(1+\nu)^{\left(\floor{\frac{t}{m}}-1\right)}}r_t  \qquad \text{for all}~~t \geq m,
    \end{align*}
    where
    \begin{align*}
        c = 1- \frac{1}{m}\left(1-\left(\frac{1}{2(1+\nu)}\right)^{1+\nu}\right).
    \end{align*}
\end{proof}

Combining the results of Theorem \ref{thm:Group-1-1} and  \ref{thm:Group-1-2}, we finish the proof of Theorem \ref{thm:Group-1}.

\subsection{Proof of Corollary \ref{cor:ignk}}

\begin{proof}
Denote $m=\ceil{n/k}$, according to Theorem \ref{thm:Group-1}, we have
\begin{align*}
r_{t+1}\leq  c^{(1+\nu)^{\left(\floor{{t}/{m}}-1\right)}}r_t
\qquad\text{with}\qquad
c = 1- \frac{1}{m}\left(1-\left(\frac{1}{2(1+\nu)}\right)^{1+\nu}\right).
\end{align*}
for all $\nu\in(0,1]$.
Noticing that the value of $c$ is monotonically decreasing according to $\nu$, we have
\begin{align*}
    1 - \frac{1}{2m} > c \geq 1 - \frac{15}{16m},
\end{align*}
which implies
\begin{align*}
    r_{t+1} \leq \Big(1 - \frac{1}{2m}\Big)^{(1+\nu)^{(\floor{t/m}-1)}}r_t
\end{align*}
for all $t\geq m$.

According to the definition of $\{r_t\}_{t\geq 0}$ and Theorem \ref{thm:Group-1-2}, 
we have
\begin{align*}
    r_0 = \max\{\Norm{\vx^0 - \vx^*},1\} \geq 1.
\end{align*}
Combining with Lemma \ref{le:n-qua-con}, we have
\begin{align*}
r_t = & a_t(m, \nu) r_0 \\
\leq & \frac{1}{2(1+\nu)}(a_{t-m}(m, \nu))^{1+\nu}r_0 \\
=& \frac{1}{2(1+\nu)r_0^{\nu}}(a_{t-m}(m, \nu))^{1+\nu}r_0^{1+\nu} \\
=& \frac{1}{2(1+\nu)r_0^\nu}r_{t-m}^{1+\nu} \\
\leq & \frac{1}{2(1+\nu)}r_{t-m}^{1+\nu}
\end{align*}
for all $t \geq m$.
This leads to 
\begin{align*}
    r_t \leq \frac{1}{4} r_{t-m}^{2}
\end{align*}
in the case of $\nu=1$.
\end{proof}

\newpage

\end{document}